\documentclass[a4paper]{article}
\usepackage{amsmath,amssymb,amsthm,epsfig}
\usepackage[utf8]{inputenc}
\usepackage{xcolor}
\usepackage[pdftex,colorlinks]{hyperref}
\usepackage{soul}

\usepackage[all]{xy}

\makeatother

\newtheorem{theor}{theorem}[section]
\newtheorem{theorem}[theor]{Theorem}
\newtheorem{defi}[theor]{Definition}
\newtheorem{defis}[theor]{Definitions}

\newtheorem{prop}[theor]{Proposition}
\newtheorem{corollary}[theor]{Corollary}

\newtheorem{remark}[theor]{Remark}
\newtheorem{lemma}[theor]{Lemma}

\newtheorem{notation}[theor]{Notation}

\newcommand{\N}{\mathbf{N}}
\newcommand{\Z}{\mathbf{Z}}
\newcommand{\R}{\mathbf{R}}

\newcommand{\f}{\rightarrow}

\newcommand{\id}{\operatorname{id}}

\newcommand{\Ker}{\operatorname{Ker}}

\newcommand{\Vol}{\operatorname{Vol}}

\newcommand{\Ent}{\operatorname{Ent}}
\newcommand{\Min}{\operatorname{Min}}
\newcommand{\Max}{\operatorname{Max}}

\newcommand{\diam}{\operatorname{diam}}

\newcommand{\g}{\operatorname{\gamma}}
\newcommand{\e}{\operatorname{\varepsilon}}

\title{Finiteness Theorems for Gromov-Hyperbolic Spaces and Groups}
\author{G.~Besson, G.~Courtois, S.~Gallot, A.~Sambusetti}

\begin{document}

\maketitle

\abstract{In this article we prove that the set of torsion-free groups acting by isometries on a hyperbolic metric space whose entropy is bounded above and with a compact quotient is finite. The number of such groups can be estimated in terms of the hyperbolicity constant and of an upper bound of the entropy of the space and of an upper bound of the diameter of its quotient. As a consequence we show that the set of non cyclic torsion-free $\delta$-hyperbolic marked groups  whose entropy is bounded above by a number $H$ is finite with cardinality depending on $\delta$ and $H$
 alone. From these results, we draw homotopical and topological finiteness theorems for compact metric spaces and manifolds.

\tableofcontents

\section{Introduction and definitions}
Finiteness results in Riemannian geometry have a long history since Weinstein's  and  Cheeger's celebrated theorems \cite{Wei}, \cite{Ch}.  The following one, which is close to the purpose of this paper, was stated  without proof by Gromov in \cite{Gr5}:

\begin{theorem}\label{thmgromov} There exist only finitely many compact, differentiable $n$-manifolds $Y$ admitting a Riemannian metric with sectional curvature $-a^2 \leq k(Y) <0$ and $diam(Y)\leq D$.
\end{theorem}

More precisely, Gromov attributed the homotopy version of this theorem to E. Heintze. Actually, Heintze proved in his dissertation \cite{Hei2, Hei} an analogous of the famous Margulis' lemma  for Hadamard manifolds  with pinched sectional curvature $-a^2 \leq k(X) \leq -b^2<0$, then deducing the finiteness of homotopy types, from Cheeger's Finiteness theorem. However, Cheeger's  theorem (besides being originally stated for simply connected manifolds) proved only finiteness up to homeomorphisms in dimension 4.
To find a complete proof of Theorem \ref{thmgromov} in the literature it seems that we had to wait for Peters's proof of Cheeger's Finiteness theorem \cite{peters} or Gromov approach through his celebrated Compactness theorem (see \cite{Gr1}), as completed by Katsuda, Greene-Wu, Peters  and others.

The notion of \emph{$\delta$-hyperbolic space}\footnote{In this article, a $\delta$-hyperbolic  space will be always supposed to be a {\em proper, geodesic} space, see definition \ref{hypdefinition0}.} which is reminiscent of the above theorem
was introduced by M. Gromov as a very weak metric notion of negative 
curvature, at some macroscopic scale. For instance, for a metric  space, being $\delta$-hyperbolic is a much weaker hypothesis than being ${\rm CAT}(-1)$
(see Proposition 1.4.3 page 12 of \cite{CDP}). Furthermore, even if the metric space is a Riemannian manifold, being $\delta$-hyperbolic gives no information on the topology or on the geometry of balls of 
small radii.

We here present several finiteness results, the proof of which relies partly on a new \emph{Bishop-Gromov Inequality}.
It may indeed seem paradoxical that an hypothesis which generalizes the notion of \lq \lq curvature bounded from above" would provide a Bishop-Gromov inequality, 
while this inequality is usually the consequence of an hypothesis of the type \lq \lq (Ricci) curvature bounded from below". However, in the sequel as well as in our previous works (see \cite{BCGS}, \cite{BG2}), the hypothesis
\lq \lq (Ricci) curvature bounded from below" will be replaced by the much weaker assumption \lq \lq entropy bounded from above" (see Definition \ref{Entropies0}). We recall the definition of  the entropy of a metric measure space.

\begin{defi}\label{Entropies0}
The entropy of a metric measure space $ (X ,d, \mu) $ (denoted by $ \Ent (X ,d, \mu) $) is the lower limit (when $ R \to +\infty $)
of $ \dfrac{1}{R}   \ln \Big( \mu \big( B_X (x , R )\big) \Big) $. It does not depend on the choice of $ x $.
\end{defi}

This invariant, possibly infinite, gives an asymptotic, hence weak, information on the geometry of the metric space (see the comparison with other 
hypotheses in the section 3.3 of \cite{BCGS} and in the section 3.1 of \cite{BG2}). Despite its apparent weakness it becomes interesting
when there exists a group $ \Gamma $ acting properly by isometries on $ (X ,d) $ (and possibly co-compactly) and when we restrict ourselves to 
$\Gamma$-invariant measures $\mu$.
Among these $\Gamma$-invariant measures, our preferred example is the \emph{counting measure} $ \mu_{x}^{\Gamma} $ of the orbit $ \Gamma  x $ 
of a given point $ x $, defined by $ \mu^\Gamma_x = \sum_{\gamma \in \Gamma} \delta_{\gamma x}$, where $\delta_y$ is the Dirac measure at the point $y$.

Notice that, in the co-compact case, the entropy does not depend on the chosen $ \Gamma$-invariant measure, as shown by the

\begin{prop}\label{Entropies1} Let $ (X ,d) $ be a non compact metric space and $ \Gamma $ be a group acting properly and co-compactly on $ (X ,d) $ 
by isometries. For every positive $\Gamma$-invariant measure $\mu $ on $ X $ such that all balls have finite measure, one has $\Ent (X ,d, \mu) = 
\Ent (X ,d, \mu_{x}^{\Gamma}) $ for every $x \in X$.\\
If, furthermore, $ (X ,d) $ is a length space, then $ \Ent (X ,d, \mu) $ is actually a limit.
\end{prop}

This proposition is classical, a proof may be found in \cite{BCGS}, Proposition 3.3. It is the reason why we shall sometimes use the notation $ \Ent (X ,d) $ instead of $ \Ent (X ,d, \mu) $ for $\Gamma$-invariant measures.

Let us now consider a finitely generated group $\Gamma$ and let $\Sigma$ be a finite and symmetric generating set (see Definition \ref{defi:Cayley}). The couple $(\Gamma, \Sigma)$ is here called a \emph{marked group} and from this data we can define a proper metric space, the so-called \emph{Cayley graph of $(\Gamma, \Sigma)$}, denoted by $ {\cal G} (\Gamma , \Sigma) $ and endowed with a natural distance such that the edges have length one (see the beginning of next Section for a precise definition). Notice that, by construction, $\Gamma$ acts isometrically and co-compactly on $ {\cal G} (\Gamma , \Sigma)$. We can then define the entropy of this metric space endowed with the counting measure and we shall denote it by $\Ent (\Gamma , \Sigma)$. When, for a finitely generated group $\Gamma$, there exists a finite generating set $\Sigma$ so that the metric space $ {\cal G} (\Gamma , \Sigma)$ is $\delta$-hyperbolic for some $\delta$ we say that $\Gamma$ is hyperbolic in the sense of Gromov and that $(\Gamma , \Sigma)$ is $\delta$-hyperbolic. We recall that the number $\delta$ depends on $\Sigma$.

One of our main results is the following theorem (see Theorem \ref{Nbgroupeshyp} and Section \ref{finiteisomorph} for the definition of an isometry between marked groups).
\begin{theorem}\label{MainIntro} Given any $\delta, H > 0$, the set of non cyclic torsion-free $\delta$-hyperbolic marked groups
$(\Gamma , \Sigma)$ satisfying $\Ent (\Gamma , \Sigma) \le H$ is finite and has cardinality (up to isometric isomorphism) smaller than a number depending on $\delta$ and $H$ only.
\end{theorem}
This number is made precise in Section \ref{sec:finiteness}. The number of finitely generated groups is infinite. Even if we bound the cardinality of $\Sigma$ it is impossible to bound the number of marked groups $(\Gamma, \Sigma)$. We need an extra piece of information and it is the size of the relations. More precisely, for a \emph{finitely presented} group $\Gamma$, the normal subgroup (of the free group $\mathbf F (\Sigma)$) generated by the relations is finitely generated (as a normal subgroup) and, if we call $\mathcal R$ one of its generating sets.  $\langle \Sigma | {\mathcal R} \rangle$ is called a \emph{presentation of $\Gamma$ by generators and relations}, both $\Sigma$ and $\mathcal R$ being finite (see Section \ref{finiteisomorph} for more details). Notice that every element $r\in {\mathcal R}$ is a word in the elements of $\Sigma$. Now an easy lemma (see Lemma \ref{presentations}) asserts that, if we consider finitely presented groups which admit a presentation $\langle \Sigma | {\mathcal R} \rangle$  such that the cardinality of $\Sigma$ is bounded above by $N\in\textbf{N}$, and if the word-length of any element of $\mathcal R$ is bounded above by $p\in\textbf{N}$, then the number of such groups is bounded above by an explicit function of $p$ and $N$.

Now a remarkable, yet easy, fact is that  every $\delta$-hyperbolic marked group $(\Gamma ,  \Sigma)$ admits a presentation $\langle \Sigma | {\mathcal R} \rangle$ by generators and relations such that the word-length of every relation $r \in {\mathcal R}$ is bounded from above by $4 \delta + 6$ (see for example \cite{BH}, chapter III.$\Gamma$, proof of Proposition 2.2). We then only need to find a good generating set $\Sigma$ for which we can control the cardinality in terms of the data $\delta$ and $H$. This is given by the following proposition (see Proposition \ref{Nbgenerators}).

\begin{prop}\label{intro:Nbgenerators}
Given $\delta, H, D > 0$, let $\Gamma$ be a non-cyclic torsion-free group acting properly and by isometries on a $\delta$-hyperbolic space $(X,d)$ verifying $\Ent (X, d) \le H$ 
and $\diam (\Gamma \backslash X) \le D$  then, for every $x \in X$, the set $\Sigma_{10 (D + \delta)}(x):=\{ \gamma \in \Gamma  : d(x , \gamma x) \le 10 (D + \delta)\}$ is a symmetric generating set
whose cardinality is bounded above by a function of $\delta$, $H$ and $D$ only. 
\end{prop}

Here the function depending on $\delta$, $H$ and $D$ only is described in Subsection \ref{subsec:finitenessaction}. This yields Theorem \ref{MainIntro} since a $\delta$-hyperbolic marked group acts on its Cayley graph co-compactly and the quotient has diameter one.

The above Proposition, when $(X,d)$ is the Cayley graph of a given $\delta$-hyperbolic marked group $(\Gamma, \Sigma)$, might be compared with a well-known  estimate, due to Arzhantseva and Lysenok \cite{Arz}: the cardinality of any other generating set $\Sigma'$ of $\Gamma$ or of any non-elementary subgroup $\Gamma'\subset\Gamma$  can be bounded in terms of its entropy
$$\# \Sigma' \leq c \cdot \Ent(\Gamma, \Sigma').$$
The main difference with our result is that Arzhantseva-Lysenok's constant $c$ depends on the given hyperbolic marked group $(\Gamma, \Sigma)$ (namely, not only on its hyperbolicity constant $\delta$, but also on the cardinality of $\Sigma$). Our estimate is universal and bounds, by the same constant,  the cardinality of $\Sigma$ for every $\delta$-hyperbolic group $(\Gamma, \Sigma)$ with entropy less than $H$. Notice that our result does not cover Azhantseva and Lysenok's.

Another related result is given in the recent article by K.~Fujiwara and Z.~Sela, \cite{FS}. In this paper the authors consider a Gromov-hyperbolic group $\Gamma$ and sequences of generating sets $\Sigma_n$ of fixed cardinality and bounded entropy. It is shown that, when it converges, the sequence $\Ent (\Gamma, \Sigma_n)$ is non decreasing (except for a finite number of indices $n$) and they show that the limit entropy is the entropy of some marked group. The results in \cite{Arz}, in \cite{FS}  and the ones in the present article give a fairly complete panorama of the situation.

Notice that Proposition \ref{intro:Nbgenerators} gives also a finiteness theorem for groups acting on a $\delta$-hyperbolic space. More precisely we also show the following result (see Theorem \ref{boundongroups}).

\begin{theorem}\label{intro:boundongroups}
Given $\delta, H, D > 0$, let ${\cal H} (\delta, H, D)$ be the set of torsion-free groups which admit a proper isometric action on some $\delta$-hyperbolic space 
$(X,d)$ satisfying $\Ent (X , d) \le H$ and $\diam (\Gamma \backslash X) \le D$. Then, ${\cal H} (\delta, H, D)$ has cardinality (up to isomorphisms) 
bounded above by a function of $\delta$, $H$ and $D$ only.
\end{theorem}
Comparing with the aforementioned finiteness Theorem \ref{thmgromov} for Riemannian manifolds that is, considering $G=\pi_1(Y)$, acting by deck transformations on the Riemannian universal covering $X$  of $Y$, we see that, besides the wider generality, in Theorem \ref{intro:boundongroups} the infinitesimal conditions given by the  upper and  lower  sectional  curvature bounds are replaced, respectively,  by the macroscopic condition of $\delta$-hyperbolicity of the universal covering, and by the asymptotic bound on the growth of $X$ given by the entropy (which is a much weaker condition, for instance, than a lower bound on the Ricci curvature). Going further in this comparison, one might ask if the $\delta$-hyperbolicity condition is really needed for Theorem \ref{intro:boundongroups}, or if it could be replaced by the Gromov hyperbolicity of $X$. We will show, in Section \ref{sec:exa}, that all the assumptions of Theorems \ref{MainIntro} and \ref{intro:boundongroups} are necessary, including the upper bound on the hyperbolicity constant.

The proof of Proposition \ref{intro:Nbgenerators} 
uses our Bishop--Gromov Inequality (see \cite{BCGS}, Theorem 5.1 and \ref{cocompact2}) as well as a result due to E.~Breuillard, B.~Green and T.~Tao (see \cite{BGT}, Corollary 1.7 and Theorem \ref{BGTbasic} in the sequel). 
The above results also rely on the key theorem stated below which shows that, in the situation under consideration, the asymptotic displacement of any hyperbolic isometry is bounded away from zero, by a positive universal constant only depending on $\delta$, $H$ and $D$. The asymptotic displacement $\ell (\g)$ of any hyperbolic isometry is defined in Definitions \ref{deplacements}. We then show the following (see Theorem \ref{theo:minordeplacement}):

\begin{theorem}\label{intro:minordeplacement}
Given $\delta , H , D > 0$, consider a $\delta$-hyperbolic metric space $(X , d) $ and a proper action by isometries of a non virtually cyclic group 
$\Gamma$ on $(X,d)$ such that the diameter of 
$\Gamma \backslash X$ and the entropy of $(X , d) $ are respectively bounded from above by $D$ and $H$, every torsion-free element $\g \in \Gamma^*$
verifies $\ell (\g) > \dfrac{2 (5 D + \delta)}{\nu \big(3^5 e^{73 H (D + 4 \delta)}\big) + 2}$.
\end{theorem}

\noindent Here $\nu(.)$ is a function which comes from the above-mentioned works by E.~Breuillard, B.~Green and T.~Tao (see \cite{BGT}, Corollary 1.7, stated as Theorem \ref{BGTbasic} in the sequel). This is a key step of this article with the most involved proof. 

In Section \ref{sec:topo}, various other finiteness results are drawn from Theorems \ref{intro:boundongroups} and \ref{intro:minordeplacement}. More precisely:

\noindent in Theorem \ref{theo:quasiisometric} we show that the number of $\delta$-hyperbolic spaces $X$ which admit a discrete torsion-free,  non-elementary  cocompact group of isometries $\Gamma$ with $\diam(\Gamma\backslash X)\leq D$ and $\Ent(X)  \leq H$ is finite, up to  quasi-isometries whose constants are computable in terms of $\delta, D, H$.

\noindent  Theorem \ref{3manifolds} implies in particular that the number of closed Riemannian 3-manifolds $Y$ with torsion-free fundamental group and with $\diam(Y) \leq D$, whose universal covering $\widetilde Y$ is $\delta$-hyperbolic and satisfies $\Ent(\widetilde Y)  \leq H$, is finite up to diffeomorphisms.
 
\noindent A consequence of Theorems \ref{boundedpi1s} and \ref{boundedtopology} is that the number of compact connected aspherical ANR metric  spaces $Y$  (respectively of closed aspherical Riemannian manifolds $Y$ of dimension $\neq 4$)    with $\diam(Y)\le D$ and  whose metric  universal covering $\widetilde Y$ is  $\delta$-hyperbolic and satisfies  $\Ent(\widetilde Y) \le H$, is finite up to homotopy equivalences (resp. up to diffeomorphisms).

Finally, in Section \ref{sec:exa} a series of examples (and counter-examples) are exhibited showing that all hypotheses of Theorems \ref{MainIntro} and \ref{intro:boundongroups} are necessary. Appendices \ref{appendices} present the basic facts about Gromov hyperbolicity in order to fix notations and conventions.

\begin{notation} In the sequel, in a metric space $(X,d)$, when there is no ambiguity on the choice of the metric $d$, we shall denote by $B_X (x,r) $  (resp. by
$\overline B_X (x,r)$) the open (resp. closed) ball of radius $r$ centred at a point $x \in X$. We shall also only consider proper geodesic spaces, that is metric spaces for which the closed balls are compact. The action by isometries of a group $\Gamma$ on $(X,d)$ is said to be {\em proper} when, for every $R> 0$, the set $\{ \g \in \Gamma : d( x , \g x) \le R\}$ is finite (this does not depend on the choice of $x \in X$).
\end{notation}

\section{Bounding the number of groups}\label{finiteisomorph}

\emph{In this section we give the definition of a marked group and state a criterion for a family of such groups to be finite, up to isomorphisms.}

\begin{defi}\label{defi:Cayley} 
A \emph{marked group} $(\Gamma, \Sigma)$ is a finitely generated group $\Gamma$ together with one of its finite generating sets $\Sigma$, that we shall always 
suppose to be symmetric\footnote{A subset $A$ of a group $\Gamma$ is said to be \lq \lq symmetric" if, for every $\g \in A$, $\g^{-1} \in A$.}. 
\end{defi}
The group $\Gamma$ is canonically 
endowed with the \emph{algebraic distance} $d_\Sigma$ associated to the generating set $\Sigma$, defined as follows: considering the action of $ \Gamma $ on its {\em Cayley graph $ {\cal G} (\Gamma , \Sigma) $}, we denote by $ | \gamma |_\Sigma $ the {\em word metric} related to $ \Sigma $, i.e the minimal number
of factors in a decomposition of $\g$ as a product of elements of $\Sigma$, we then define $ d_\Sigma $ by 
$$ d_\Sigma (\g, g) := |\g^{-1} g   |_\Sigma\,. $$

By definition, the algebraic distance $d_\Sigma $ on $\Gamma$ is the restriction to the vertices of the length distance on the graph $ {\cal G} (\Gamma , \Sigma) $. 
For the sake of simplicity, we shall use the same notation $ d_\Sigma $ for these two distances, on the group and on the graph.

An {\em isomorphism of marked groups} $\varphi : (\Gamma, \Sigma) \f (\Gamma', \Sigma')$ is an isomorphism $\varphi : \Gamma \f \Gamma'$ which maps 
$\Sigma$ onto $\Sigma'$. As the generating sets are all supposed to be symmetric, we remark that $\varphi$ is an isomorphism of marked groups if and only if it is an isometric isomorphism between $(\Gamma, d_\Sigma)$ and $(\Gamma', d_{\Sigma'})$.
For the sake of simplicity, we shall often write that two marked groups are \lq \lq {\em isometric}" instead of writing that they are \lq \lq isomorphic by an isomorphism of marked groups". Hence, by definition, an {\em isometry} between $(\Gamma, \Sigma)$ and $(\Gamma', \Sigma')$ will always be assumed
to be an isomorphism of marked groups.

Now let $\mathbf F (\Sigma)$ be \emph{the free group generated by $\Sigma$} {\sl i.e.}, if $\# \Sigma =k$, the symmetric generating set of $\mathbf F (\Sigma)$ has $2k$ elements and let $\varphi_{\Sigma} : \mathbf F (\Sigma) \f \Gamma$ be \emph{the canonical epimorphism} which maps $\Sigma\subset\mathbf F (\Sigma) $ onto $\Sigma\subset\Gamma$. Obviously $\varphi_{\Sigma}$ induces an isomorphism between $\mathbf F (\Sigma) / {\Ker \varphi_{\Sigma}}$ and $\Gamma $, called \emph{the canonical isomorphism}. The normal subgroup $\Ker \varphi_{\Sigma}\subset\mathbf F (\Sigma)$ is called the group of relations between elements of $\Sigma$. 
For any subset $\mathcal R$ of $\Ker \varphi_{\Sigma}$ 
such that $\Ker \varphi_{\Sigma}$ is the smallest normal subgroup of $\mathbf F (\Sigma)$ which contains $\mathcal R$, we shall say that $\mathcal R$ is a 
\emph{generating set of the relations}  and that $\langle \Sigma | {\mathcal R} \rangle$ is a \emph{presentation of $\Gamma$ by generators and relations}.
Let us insist on the fact that any relation in $\Ker \varphi_{\Sigma}$ is a product of \emph{conjugates} of the elements of $\mathcal R$.

Conversely, given a finite set $\Sigma$, one can construct the associated free group $\mathbf F (\Sigma)$; to any finite subset $\mathcal R$ of 
$\mathbf F (\Sigma)$ corresponds a group $\Gamma$ whose presentation is $\langle \Sigma  | {\mathcal R} \rangle$. Indeed, let $\langle\langle {\mathcal R} \rangle\rangle$ be the
smallest normal subgroup of $\mathbf F (\Sigma)$ containing $\mathcal R$, we define $\Gamma$ as $\mathbf F (\Sigma) / \langle\langle {\mathcal R} \rangle\rangle$.

\begin{notation} 
For the sake of simplicity, we shall say that a family $\mathcal F$ of marked groups has cardinality (up to isometries) smaller than $q < + \infty$
whenever there exist marked groups $(\Gamma_1, \Sigma_1), \ldots , (\Gamma_q, \Sigma_q)$ in $\mathcal F$ such that any other marked group $(\Gamma , \Sigma)\in \mathcal F$ is isometric to some of the $(\Gamma_i, \Sigma_i)$'s. We shall also say that a family $\mathcal G$ of groups has cardinality (up to isomorphisms) smaller than $q$
whenever there exist groups $\Gamma_1 , \ldots , \Gamma_q $ in $\mathcal G$ such that any other group $\Gamma \in \mathcal G$ is isomorphic to some of the $\Gamma_i $'s.
\end{notation}

The following criterion for finiteness is part of the folklore on the subject:

\begin{lemma}\label{presentations}
Given integers $N, p$ such that $N\ge 1$ and $p \ge 3$, let us define $q := \sum_{k = 0}^{N}  2^{(2 k)^p}$. Then, the family of marked groups $(\Gamma , \Sigma )$, such that the group $\Gamma$ admits a presentation $\langle \Sigma | {\mathcal R} \rangle$ by generators and relations with $\# \Sigma \le N$ and such that the 
word-length $| r |_\Sigma\leq p$ for every relation $r \in {\mathcal R}$, has cardinality (up to isometries) smaller than $q < + \infty$.
\end{lemma}

\begin{proof} 
Given a set $\Sigma$ such that $1 \le k := \# \Sigma \le N$, the number of possible presentations $\langle \Sigma | {\mathcal R} \rangle$ of non trivial groups, 
such that the word-length of every $r \in {\mathcal R}$ is bounded from above by $p$, is smaller than the number of subsets of the closed ball of radius $p$ in 
$\mathbf F (\Sigma)$. It is thus smaller than $2^{(2 k)^{p}}$, since the number of elements in this ball  is 
$$1 +2 k + 2 k (2 k-1) \ldots +2 k (2 k-1)^{p-1} \le (2 k)^p\,, $$
the proof of this last inequality being left to the reader (Hint: use the hypothesis 
$p \ge 3$ to solve the case $k \ge 2$).

For every presentation $\langle \Sigma' | {\mathcal R}' \rangle$ of a marked group $(\Gamma ', \Sigma ')$ satisfying the assumptions of the lemma, {\sl i.e} such that $\# \Sigma' = \# \Sigma =k \leq N$ and $|r|_{\Sigma '} \leq p$ for every $r\in \mathcal R '$, we choose a marked group isomorphism 
$\psi : (\mathbf F (\Sigma) , \Sigma) \f \mathbf (F (\Sigma'), \Sigma ')$ and define ${\mathcal R} := \psi^{-1} ({\mathcal R}')$. Notice that every $r\in \mathcal R$ satisfies $|r|_\Sigma \leq p$ since every element in $\mathcal R '$ does. We obtain that
$\langle\langle {\mathcal R}' \rangle\rangle = \psi \big(\langle\langle {\mathcal R} \rangle\rangle \big)$ and $\psi$ induces an 
isomorphism $\widetilde \psi : \mathbf F (\Sigma) / \langle\langle {\mathcal R} \rangle\rangle \f \mathbf F (\Sigma') / \langle\langle {\mathcal R}' \rangle\rangle$.
It follows that $\widetilde \psi$ is an isomorphism of marked groups between $\big(\mathbf F (\Sigma) / \langle\langle {\mathcal R} \rangle\rangle , \Sigma \big)$
 and
$\big(\mathbf F (\Sigma') / \langle\langle {\mathcal R}' \rangle\rangle , \Sigma' \big)$.
Hence, for every marked group $(\Gamma ', \Sigma ')$ such that $\#\Sigma ' =k$, there exists some subset $\mathcal R$ of the closed ball of radius $p$ in $\mathbf F (\Sigma)$ such that 
$(\Gamma ', \Sigma ')$ is isomorphic to $\big(\mathbf F (\Sigma) / \langle\langle {\mathcal R} \rangle\rangle , \Sigma \big)$, where $\#\Sigma  =\#\Sigma' \leq N$ and every $r\in \mathcal R$ satisfies $|r|_\Sigma \leq p$.
Therefore, for every $k \ge 1$, the number of marked groups $(\Gamma ', \Sigma ')$  (modulo isomorphisms of marked groups, {\sl i.e.} up to 
isometries), such that $\# \Sigma ' = k$ and admitting a presentation $\langle \Sigma' | {\mathcal R'} \rangle$ with relators of length at most $p$, is smaller than 
the number of possible presentations $\langle \Sigma | {\mathcal R} \rangle$ as above, and consequently is less than $2^{(2 k)^{p}}$. Summing in $k$ ends the proof of the lemma.
\end{proof}

\begin{remark}
As a conclusion, if one aims at bounding the cardinality (up to isometries) of a family $\mathcal F$ of marked groups, the goal is then clear: one 
has to prove that, for every $(\Gamma , \Sigma ) \in {\cal F}$, the cardinality of its set of generators $\Sigma$ is uniformly bounded and to find 
a generating set $\mathcal R$ of its relations such that the word-length of all the elements of $\mathcal R$ is uniformly bounded too.
\end{remark}

In Subsections \ref{subsec:simplyconnected} and  \ref{subsec:BoundingHyperbolic}  we present two situations where it is possible to bound the word-length of all the elements of a generating set of the relations.

\subsection{Groups acting on a simply connected metric space}\label{sec:simplyconnected}\label{subsec:simplyconnected}

In this subsection, we consider a path-connected metric space $(X , d) $, and a proper action by isometries of a group $\Gamma$, on 
$(X,d)$, such that the diameter of $\Gamma \backslash X$ is bounded from above by $D$. Choosing $x \in X$ and a real number $R>0$, we denote by $\Sigma_R(x):=\{ \gamma \in \Gamma  : d(x , \gamma x) \le R\}$ and by $\Gamma_R(x)$ the subgroup of $\Gamma$ generated by $\Sigma_R(x)$. Let $k\geq 2$ be an integer, the set $\Sigma:=\Sigma_{kD}(x)$ is a symmetric generating set of $\Gamma$ by Proposition \ref{generatorsdiameter}, that is: $\Gamma_{kD}(x)=\Gamma$.

When $(X,d)$ is a simply connected length space, each group $\Gamma$ admits a presentation $\langle \Sigma_{2 D} (x) |  {\mathcal R} \rangle$ 
such that the word-length of all the elements of $\mathcal R$ is $\le 3$. This is an immediate consequence of the following Lemma, which can be found in 
J. P. Serre's book (\cite{Se} p. 30), citing A. M. Macbeath\footnote{\cite{Se} also refers to analogous results by Gerstenhaber, Behr and Weil.}, and see M. Gromov (\cite{Gr1}) for a Riemannian version:

\begin{lemma}\label{Serre} Let $X$ be a path-connected topological space and $\Gamma$ a group acting by homeomorphisms on $X$. Then, for every path-connected open set $U \subset X$ such that  $ \cup_{\g \in \Gamma} \, \g (U) = X$, the set $\Sigma := \{ \g \in \Gamma : \g(U) \cap U \ne \emptyset\} $ is a symmetric generating set of $\Gamma$. If moreover $X$ is simply connected, then there exists a generating set $\mathcal R$ of the relations between elements of $\Sigma$ such that each element of $\mathcal R$ can be written as a product 
$ s\, \sigma \, t$ where $s ,\sigma , t \in \Sigma$. More precisely $\mathcal R$ is defined as the set of the $ s \, \sigma \, t \in \mathbf F (\Sigma)$ such that 
$U \cap \varphi_\Sigma (s).U \cap \varphi_\Sigma (s \sigma).U \ne \emptyset$ and $ \varphi_\Sigma (t) =  \varphi_\Sigma \big( (s \sigma)^{-1}\big)$.
\end{lemma}

A consequence is the following 

\begin{corollary}\label{Serre1}
Let $N\ge 1$ be an integer, we define $q := \sum_{k = 0}^{N}  2^{(2 k)^3}$. The family of groups $\Gamma$ which admit a proper isometric action on a simply connected length space $(X,d)$ such that 
$\diam (\Gamma \backslash X) \le D$ and $\# \big( \Sigma_{2 D} (x) \big) \le N$ for at least one $x \in X$, has cardinality (up to isomorphisms) smaller than $q$.
\end{corollary}

\begin{proof}[Proof of Corollary \ref{Serre1}]
For every $\e> 0$, we apply Lemma \ref{Serre}, choosing $U := B_X (x, D +\e)$. Indeed, $B_X (x, D+\e)$ is  path-connected since $(X,d)$ is a 
length space and, as $B_X (x, D + \e)$ contains a fundamental domain, $ \cup_{\g \in \Gamma} \, \g \big(B_X (x, D+ \e) \big) = \cup_{\g \in \Gamma} \, B_X (\g x, D+ \e) = X$. As $(X,d)$ is a length space, the condition $B_X ( x, D+ \e) \cap \g \big(B_X (x, D+ \e) \big) \ne \emptyset $ is equivalent to 
$d (x, \g x) < 2 (D+ \e)$. Lemma \ref{Serre} then proves that $\Sigma := \{ \g \in \Gamma : \g(U) \cap U \ne \emptyset\} 
=  \Sigma_{2 (D+\e)} (x) $  is a symmetric generating set of $\Gamma$ and that there exists a generating set $\mathcal R$ of the relations between elements of $\Sigma$ such that each element of $\mathcal R$ can be written $ s\, \sigma \, t$ where $s ,\sigma , t \in \Sigma$. We moreover notice that, if $\e$ is small 
enough, then $\Sigma_{2 (D+\e)} (x) = \Sigma_{2 D} (x)$; indeed, the properness of the action implies that $\Sigma_{2 (D+\e)} (x)$, hence $\Sigma_{2 (D+\e)} (x)\setminus\Sigma_{2 D} (x)$, is finite.

It follows that the hypotheses of Lemma \ref{presentations} are verified by the marked group $(\Gamma, \Sigma)$, where $\Sigma := \Sigma_{2 D} (x)$, 
where $p = 3$ and $q := \sum_{k = 0}^{N}  2^{(2 k)^3}$, Lemma \ref{presentations} then proves the finiteness, up to isometries, of the marked groups $(\Gamma, \Sigma_{2 D} (x))$, hence the finiteness up to isomorphisms of the groups $\Gamma$.
\end{proof}

\subsection{Bounding hyperbolic groups}\label{subsec:BoundingHyperbolic}

The goal of this article is to bound the number of $\delta$-hyperbolic groups satisfying certain properties (see Theorem \ref{Nbgroupeshyp} for a precise statement). The following result is a first step.

\begin{prop}\label{Nbgroupes}
For every $k\in \N^*$ and $\delta \geq 0$, let $q_k := \sum_{i = 0}^{k}  2^{(2 i)^{4\delta + 6}}$. The set  of $\delta$-hyperbolic marked groups $(\Gamma , \Sigma)$ such that $\# \Sigma \le k$ has cardinality (up to isometries) less than $q_k$.
\end{prop}

\begin{proof} It is well known (see for example \cite{BH}, chapter III.$\Gamma$, proof of Proposition 2.2), that every $\delta$-hyperbolic marked group 
$(\Gamma ,  \Sigma)$ admits a presentation $\langle \Sigma | {\mathcal R} \rangle$ by generators and relations such that the word-length of every relation $r \in {\mathcal R}$ 
is bounded from above by $4 \delta + 6$. We then apply Lemma \ref{presentations}, where we replace $N$ by $k$ and $p$ by $4 \delta + 6$.
\end{proof} 

\begin{remark}\label{rem:bound}
Consequently, if one aims at bounding the number (up to isometries) of $\delta$-hyperbolic marked groups $(\Gamma , \Sigma)$, belonging to a family $\mathcal F$, the goal is then clear: one should bound from above the cardinality of their generating sets $\Sigma$. 
\end{remark}

\section{A lower bound for the asymptotic displacement}\label{systolebound2}

\emph{In order to prove the main theorem (Theorem \ref{MainIntro}) we shall make sure that there is no sequence of spaces $(X, d)$, with groups $\Gamma$ acting properly by isometries on it, satisfying the assumptions, and with elements $\gamma\in\Gamma$ whose displacement at some point $x\in X$ goes to zero. Theorem \ref{theo:minordeplacement} below is the key of the proof of Theorem \ref{MainIntro}. It relies on two other results which we state before.}

\subsection{Main Tools}

It was a revolution when M.~Gromov extended the classical Bishop's inequality to what is now called the Bishop-Gromov's inequality and used it to interpret Cheeger's finiteness Theorem as a consequence of a compactness result. Extending  Bishop-Gromov's inequality to metric measure spaces is a challenging question. The following theorem is such an extension to $\delta$-hyperbolic spaces endowed with a co-compact action of some group of isometries. It is proved in \cite{BCGS}.

\begin{theorem}\label{cocompact2} \emph{(\cite{BCGS}, Theorem 5.1)}
Let $(X , d) $ be a $ \delta$-hyperbolic metric space, for every proper action 
by isometries of a group $\Gamma$ on $(X,d)$ such that the diameter of $\Gamma \backslash X$ and the
entropy of $(X , d) $ are respectively bounded from above by $D$ and $H$, then, for all $x \in X$ and all $R > r \ge 10\,(D + \delta )$, 
the counting measure $ \mu^\Gamma_x $ of the orbit $\Gamma x$ verifies the inequality:
$$\dfrac{\mu_{x}^{\Gamma} \left( B_X \big( x , R \big)\right)}{\mu_{x}^{\Gamma} \big( B_X ( x , r )\big)} < 3 \left( \frac{R}{r} \right)^{25/4} e^{6 H (R- \frac{4}{5} r)} \, .$$
\end{theorem}

\begin{remark} 
In Theorem \ref{cocompact2} the inequality itself only depends on the upper bound of the entropy. The upper bounds of the hyperbolicity 
constant $\delta$ and of the diameter $D$ enter only in the computation of the minimal value of the radii of the balls which verify the Bishop-Gromov's inequality.
\end{remark}

We shall combine Theorem  \ref{cocompact2}  with the following result due to E.~Breuillard, B.~Green and T.~Tao (see \cite{BGT}, Corollary 1.7), 
 where $A\cdot B$ is the set of products $a.b$, where $a \in A$ and $b \in B\,$:

\begin{theorem}\label{BGTbasic}
For every $K \ge 1$, there exists universal constants $\nu (K),\  \nu' (K) \in \N^*$ (only depending on $K$) with the following properties: if $G$ is a group and 
if  $A$ and $B$ are 
finite non-empty subsets of $G$ such that $\# (A\cdot B) \le K (\# A)^{1/2} (\# B)^{1/2}$ there then exist a subgroup $G_0$ of $G$ 
and a finite normal subgroup $L$ of $G_0$ which verify:
\begin{itemize}
\item[(i)] $A$ is covered by $\nu (K)$ left-translates of $G_0$,
\item[(ii)] $G_0/L$ is nilpotent and admits a generating set with less than $\nu' (K)$ elements.
\end{itemize}
\end{theorem}

\subsection{A lower bound for the asymptotic displacement}

The notions of asymptotic (resp. minimal) displacement $\ell (\g)$ (resp. $s (\g)$) of any hyperbolic isometry is defined in Definitions \ref{deplacements}. Still denoting by $\nu (.)$ the universal nondecreasing integer valued function defined in Theorem \ref{BGTbasic}, we have the

\begin{theorem}\label{theo:minordeplacement}
Given $\delta , H , D > 0$, consider a $\delta$-hyperbolic metric space $(X , d) $ and a proper action by isometries of a non virtually cyclic group 
$\Gamma$ on $(X,d)$ such that the diameter of 
$\Gamma \backslash X$ and the entropy of $(X , d) $ are respectively bounded from above by $D$ and $H$, then every torsion-free element $\g \in \Gamma^*$
verifies $\ell (\g) > \dfrac{2 (5 D + \delta)}{\nu \big(3^5 e^{73 H (D + 4 \delta)}\big) + 2}$.
\end{theorem}

Before proving this Theorem, we first give the following definitions and state a lemma.

\begin{defis}\label{faisceau}
On a $\delta$-hyperbolic space $(X,d )$, for an hyperbolic isometry $\g$, we denote by $ \gamma^-$ and $ \gamma^+$ the fixed points of $\g$ 
(see their definition after Theorem \ref{ellparahyp}),
by ${\cal G} (\g)$ the set of geodesic lines such that $ c(-\infty) = \gamma^-$ and $ c(+ \infty) = \gamma^+$ and by $M(\g)$ the subset of 
$X$ obtained as the union of these geodesic lines.\\
We then define $M_{\rm{min}} (\g)$ as the closed non empty\footnote{The fact that this set is non empty and, by continuity, closed is proved in \cite{BCGS}, Lemma 8.34 (iv).} set of points of $X$ where the function $ x \f d(x, \g\, x)$ attains its infimum $ s(\gamma)$.
\end{defis}

 The point (i) of the following lemma is a consequence of Lemma 8.28 of \cite{BCGS}. We reproduce it for the sake of completeness.

\begin{lemma}\label{distgeod}
If $  \ell(\g) > 3\, \delta$, then 
\begin{itemize}
\item[(i)] every point $x \in M_{\rm{min}} (\g)$ satisfies $d \big(x , M(\g) \big) \le \frac{7}{2} \delta$,
\item[(ii)] every point $x \in M (\g)$ satisfies $d \big(x , M_{\rm{min}} (\g) \big) \le  \frac{15}{2} \delta$.
\end{itemize}
\end{lemma}

\begin{proof}[Proof of Lemma \ref{distgeod}]
Let us consider a geodesic $ c \in \cal{G} (\g)$, oriented from $\g^-$ to $\g^+$, and a point $x \in M_{\rm{min}} (\g)$. For every 
$ k \in \Z$, we  denote by $c (t_k)$ a projection of the point $\g^k x$ on the image of the geodesic $c$. For every $k \in \Z$,
let $c (t_{k}')$ be a projection of $\g^k (  c (t_{0} )  )$ on the image of the geodesic $c$, then Proposition 8.10 (i) 
of \cite{BCGS} and the fact that $c$ and $ \g^k \circ \, c$ are two geodesics from $\gamma^-$ to $\gamma^+$ ensure that 
$d \left( \g^k c (t_{0} ) \,,\, c (t_{k}')\right) \le 2\, \delta$, and this yields:
\begin{equation}\label{distac}
d \left( \g^k x , c (t_{k})\right) \le  d \left( \g^k x , c (t_{k}')\right) \le d \left( \g^k x , \g^k c (t_{0} ) \right) + 
d \left( \g^k c (t_{0} ) \,,\, c (t_{k}')\right) \le d (x, c ) + 2 \delta \, .
\end{equation}
As $ \g^+$ (resp. $ \g^-$) is the limit of $\g^k x$ when $k \to +\infty$, it follows from \eqref{distac} that $c(t_k)$ goes to $ \g^+$ when $k \to +\infty$ 
and to $ \g^-$ when $k \to -\infty$; this yields
\begin{equation}\label{distac1}
t_k \to +\infty \text{ when } k \to +\infty \ \ ,\ \ t_k \to - \infty \text{ when } k \to -\infty \ \ \text{ hence } \ \ \R = \cup_{k \in \Z} [t_k , t_{k+1}] \, .
\end{equation}

Let us now prove that
\begin{equation}\label{shift}
 \forall \e > 0 \quad \exists p \in \N  \ \text{ such that } \, d \big( c (t_p) , c (t_{p+1}) \big) > \ell (\g) -  \frac{1}{2}\,\e\ .
\end{equation}
Indeed, from \eqref{distac}, we deduce that 
$$| d \big(  c (t_{0})  ,  c (t_{k}) \big) - d \big( x , \g^k x \big)| \le  d \big( x , c (t_{0})\big) + 
d \left( \g^k x , c (t_{k})\right) \le  2\, d \big(x, {\rm Im}(c) \big) + 2\, \delta\,,$$ 
thus that
$$ \lim_{k \f +\infty} \left(\frac{1}{k}\, \sum_{i = 0}^{k - 1} d \big(  c (t_{i})  ,  c (t_{i+1}) 
\big) \right) \ge \lim_{k \f +\infty} \left( \frac{1}{k}\,  d \big(  c (t_{0})  ,  c (t_{k}) \big) \right)
= \lim_{k \f +\infty} \left(  \frac{1}{k} \, d \big( x , \g^k x \big)\right) = \ell(\g) \,;$$
a consequence is that $ \sup_{p \in \N} \  d \big(  c (t_{p})  ,  c (t_{p+1}) \big) \ge  \ell(\g)$, and this proves property \eqref{shift}.

\smallskip
When $ \ell (\g) > 3 \,\delta $, choosing $\e$ small enough, property \eqref{shift} implies the existence of some $p\in \N$ such that 
$d \big(c (t_p) , c (t_{p+1}) \big) > \ell (\g) -  \frac{1}{2}\,\e > 3 \,\delta $, and, for every $k \in \Z$, it follows from this and from Lemma 
\ref{ecartement} that
$$ d \left( x ,  \g \, x\right) = d \left( \g^{p} x ,  \g^{p +1} x\right) \ge  d \left( \g^{p} x ,  c (t_p)\right) + 
d \big(c (t_p) , c( t_{p +1})\big)  + d \left( c( t_{p +1}) , \g^{p +1} x \right) - 6\, \delta$$
\begin{equation*}
>  d \left( \g^{k} x ,  \g^{k-p} c (t_p)\right) + \ell (\g) -  \frac{1}{2}\,\e  + d \left( \g^{k} x , \g^{k-p-1} c( t_{p +1}) \right) - 6\, \delta\, .
\end{equation*}
Letting $\e \to 0$  in this inequality and noticing that the geodesics $\g^{k-p}\circ \, c $ and $\g^{k - p -1}\circ \, c $  belong to $\cal{G} (\g)$ we deduce
that $ s(\g) = d \left( x ,  \g \, x\right) \ge 2 \, d \big( \g^{k} x,  M(\g) \big) + \ell(\g) - \,6\, \delta$, hence that
\begin{equation}\label{distgeod1}
\forall k \in \Z \ \ \ \  d \big(\g^{k}x, M(\g) \big) \le \frac{1}{2} \big( s (\g)- \ell(\g)\big) + 3\,\delta \le \frac{7}{2} \delta
\end{equation}
where the last inequality is a consequence of Lemma \ref{quasigeod}. This proves (i) when letting $k=0$ in \eqref{distgeod1}.\\

Since, by Proposition 8.10 (i) of \cite{BCGS}, every point $y$ of $M(\g)$ verifies $d(y, c ) \le 2 \delta$ we get that $d \big(\g^{k}x, c \big) \le 
d \big(\g^{k}x, M(\g) \big) + 2 \delta$. From this and from \eqref{distgeod1} we deduce that
\begin{equation}\label{distgeod2}
\forall k \in \Z\,, \quad  d \big(\g^{k}x, c \big) \le \frac{1}{2} \big( s (\g)- \ell(\g)\big) + 5\,\delta \le \frac{11}{2} \delta\,.
\end{equation}
Applying the convexity Lemma \ref{convexity} (ii) to the two geodesics $[\g^{k}x , \g^{k+1}x ]$ and $[c(t_k) , c(t_{k+1}]$, we obtain that, for every point
$u = c(t) \in [c(t_k) , c(t_{k+1}]$, there exists $v \in [\g^{k}x , \g^{k+1}x ] := \g^{k} ([x , \g x])$ such that 

$$ d(c(t) , v) \le \Max \big( d(\g^{k}x , c(t_k) ) \, ,\, d(\g^{k+1}x , c(t_{k+1}) ) \big) + 2 \delta= \dots$$
$$\,\,\dots =\Max \big( d(\g^{k}x , c)\, ,\, d(\g^{k+1}x , c ) \big)  + 2 \delta\le \frac{15}{2} \delta \,,$$

where the last inequality is a consequence of \eqref{distgeod2}. As, by lemma \ref{geodmin}, $ [\g^{k}x , \g^{k+1}x ] \subset M_{\rm{min}} (\g)$,
we get that $d \big( c(t) ,  M_{\rm{min}} (\g) \big) \le  \frac{15}{2} \delta $ for every $ t \in [t_k, t_{k+1}]$ and every $k \in \Z$, and then that 
$d \big( c(t) ,  M_{\rm{min}} (\g) \big) \le  \frac{15}{2} \delta $ for every $ t \in \cup_{k \in \Z} [t_k , t_{k+1}] = \R$ by \eqref{distac1}.
Since this is verified for every geodesic $ c \in \cal{G} (\g)$, one has $d \big(y , M_{\rm{min}} (\g) \big) \le \frac{15}{2} \delta $ for every point $y \in M (\g)$, 
which ends the proof of (ii).
\end{proof}

The main argument in order to prove Theorem \ref{theo:minordeplacement} is the following

\begin{prop}\label{decrireG3}
Let $(X , d) $ be a $ \delta$-hyperbolic metric space, and a group $\Gamma$ acting on $(X,d)$ properly by isometries, such that the diameter of 
$\Gamma \backslash X$ and the entropy of $(X , d) $ are respectively bounded from above by $D$ and $H$, then, for every 
$R \ge \frac{5}{2}  (D + 4\delta)$, if there exists a torsion-free element $\sigma$ of $\Gamma^*$ such that 
$\ell(\sigma) \le \dfrac{2(2 R - 19 \delta)}{\nu \big(3^5 e^{29 H R}\big) + 2}$, then $\Gamma_{R} (x)$  is virtually cyclic at every point 
$x \in M(\sigma)$.
\end{prop}

\begin{proof}
Given $\delta, H , D$ and $R \ge \frac{5}{2} (D + 4\delta)$, let us denote by $\nu$ the integer $\nu \big(3^5 e^{29 H R}\big)$, for the sake of
simplicity. Let $\sigma$ be an element of $\Gamma^*$ such that $0 <\ell(\sigma) \le \frac{2(2 R - 19 \delta)}{\nu + 2} $, choose a point 
$x \in M(\sigma)$. 
We note that $M(\sigma) =  M(\sigma^k) $ for every $k \in \Z^*$.
By Lemma \ref{distgeod} (ii), for every $k$ such that $ \ell(\sigma^k) > 3 \delta$, there exists a point $x_k \in M_{\rm{min}} (\sigma^k)$ such that 
$d(x, x_k) \le \frac{15}{2} \delta$. The triangle inequality and Lemma \ref{quasigeod} then imply that 
\begin{equation}\label{controledistance}
d(x , \sigma^k  x) \le d(x_k , \sigma^k x_k) + 15 \delta = s (\sigma^k) + 15 \delta \le \ell (\sigma^k) + 16 \delta = |k| \,\ell(\sigma) + 16 \delta\,.
\end{equation}

Let us define $A = B := \Sigma_{4 R} (x) = \{\g \in \Gamma : d(x, \g x) \le  4 R \}$ and $A\cdot B=A \cdot A :=\{ a.b : a\in A\,\,{\rm and}\,\, b\in B\}$. We wish to bound $\#A\cdot A$ using Theorem \ref{cocompact2}. 

Indeed,  for $\eta >0$, Theorem \ref{cocompact2} implies that
$$ \dfrac{\# (A\cdot B)}{(\# A)^{1/2} (\# B)^{1/2}} = \dfrac{\# (A\cdot A)}{\# A} \le
\dfrac{\mu_{x}^{\Gamma} \big( B_X ( x , 8 (R + \eta))\big)}{\mu_{x}^{\Gamma} \big( B_X ( x , 4 R )\big)} \le 3 \left( \frac{2 (R + \eta)}{R} \right)^{25/4} e^{48 H (\frac{3}{5} R + \eta)} \,.$$
Taking the limit when $\eta$ goes to zero, we deduce that 
$$\dfrac{\# (A\cdot B)}{(\# A)^{1/2} (\# B)^{1/2}} \le 3 \cdot 2^{25/4} e^{29 H R}
\le 3^5 e^{29 H R}\,.$$

We may then apply Theorem \ref{BGTbasic} with $K=3^5 e^{29 H R}$, and get the following consequences: there exist a virtually nilpotent group $G_0$, and a subset $S := \{\g_1, \ldots , \g_{\nu}\}$ of $\Gamma$, with $\nu (K) = \nu$ elements such that $\Sigma_{4 R} (x) \subset \bigcup_{i = 1}^{\nu} \g_i \cdot G_0$. Notice that $\g_i \cdot G_0$ and $\g_j \cdot G_0$ are either disjoint or equal.

Since $\sigma$ is torsion-free and the action is co-compact, the hypothesis (ii) of Lemma \ref{maximalcyclic} is verified, hence $\sigma$ is  
an hyperbolic isometry and is contained in a 
unique maximal virtually cyclic subgroup of $\Gamma$, denoted by $\Gamma_\sigma$. For every $R' > 19 \delta$, we deduce from \eqref{controledistance} 
that $ |k|\, \ell(\sigma) \le R'- 16 \delta$ implies $d(x, \sigma^k \, x)  \le R' $, which implies that $\sigma^k \in \Sigma_{R'} (x)$, this yields
\begin{equation}\label{controledistance0}
\# \left\{k : \sigma^k \in \Sigma_{R'} (x) \right\}\ge  \# \{k : 3 \delta < |k|\, \ell(\sigma) \le R'- 16 \delta \} > 2 \left( \frac{R' - 19 \delta}{\ell (\sigma)} -1 \right)\, .
\end{equation}
The hypothesis $\ell(\sigma) \le  \dfrac{2(2 R - 19 \delta)}{\nu + 2} $ and \eqref{controledistance0}, where we set $R' = 4 R$, imply that 
$$ \# \left\{k :  \sigma^k \in \Sigma_{4 R} (x) \right\} > 2 
\left( \frac{4 R - 19 \delta}{\ell (\sigma)} -1 \right) \ge  2 \left( \frac{4 R - 19 \delta}{2(2 R - 19 \delta)} (\nu + 2) -1 \right) >  \nu \, .$$
It follows that $\# \left\{k :  \sigma^k \in \bigcup_{i = 1}^{\nu} \g_i \cdot G_0 \cap \Sigma_{4 R} (x)  \right\} > \nu$, hence that 
there exist $p,q \in \Z$, with $p\ne q$, such that $\sigma^p,\, \sigma^q \in \Sigma_{4 R} (x)$ and such that $\sigma^p$ and 
$\sigma^q$ are two distinct elements of the same coset $\g_i . G_0$, for some $\g_i \in S$. A consequence is that $\sigma^{p-q}$ is a non trivial torsion-free element of $ G_0$. As $\Gamma$ acts co-compactly on a Gromov-hyperbolic space, 
the virtually nilpotent subgroup $G_0$ is virtually cyclic by Proposition \ref{actionelementaire} (iii) and \ref{actioncocompacte} 
(iv). As $G_0$ is a virtually cyclic subgroup containing $\sigma^{p-q}$ and as $\Gamma_\sigma$ is the maximal virtually cyclic subgroup containing 
$\sigma^{p-q}$, we have $G_0 \subset \Gamma_\sigma$.

For an element $g$ of $\Sigma_{R} (x) $, then $ g^{-1}\sigma \,g$ lies in the maximal virtually cyclic subgroup $g^{-1} \Gamma_\sigma \,g $.
On the other hand, the triangle inequality and the $\Gamma$-invariance of the distance give
$$  \sigma^k \in \Sigma_{2 R} (x) \implies   \sigma^k \in \Sigma_{4 R} (g x) \iff g^{-1}\sigma^k \,g \in \Sigma_{4 R} (x) \, .$$

The hypothesis on $R$ implies that $ 2 R > 19 \delta$ then, the above implication and the inequality \eqref{controledistance0}, where we set $R' = 2 R $, yield:
$$\# \left\{k :  g^{-1}\sigma^k \,g \in \Sigma_{4 R} (x) \right\} \ge  \# \left\{k : \sigma^k \in \Sigma_{2 R} (x)  \right\} > 
2 \left( \frac{2 R - 19 \delta}{\ell (\sigma)} -1 \right)  \, .$$
From this last inequality and from the hypothesis $\ell(\sigma) \le  \dfrac{2(2 R - 19 \delta)}{\nu + 2}$, it comes that 

$$\# \big\{k : g^{-1}\sigma^k \,g  \in \bigcup_{i = 1}^{\nu} \g_i \cdot G_0 \cap \Sigma_{4 R} (x)  \big\} = \# \left\{k :  g^{-1}\sigma^k \,g \in \Sigma_{4 R} (x) \right\} > \nu\,.$$

It follows that there exist $k, r \in \Z$ (with $k \ne r$) such that $g^{-1} \sigma^k g$ and $g^{-1} \sigma^r g$  belong to $\Sigma_{4 R} (x)$ and are two
distinct elements of the same coset $\g_i . G_0$, for some $\g_i \in S$. Mimicking the previous proof, we obtain that $g^{-1} \sigma^{k-r} g$ is a non trivial hyperbolic element of $G_0 \subset \Gamma_\sigma$ and, as it is also an hyperbolic element of $g^{-1} \Gamma_\sigma g $, the two maximal virtually cyclic subgroups $\Gamma_\sigma $ and 
$g^{-1} \Gamma_\sigma g $, both containing $g^{-1} \sigma^{k-r} g$, coincide by Lemma \ref{maximalcyclic}.

As $\sigma$ and $g^{-1} \sigma g$ are both contained in $\Gamma_\sigma$ and are both hyperbolic, they generate a virtually cyclic subgroup 
of $\Gamma$ and
Proposition \ref{actionelementaire} (v) then implies that $\sigma$ and $g$ generate a virtually cyclic subgroup of $\Gamma$ which is thus included in the maximal
virtually cyclic subgroup $\Gamma_\sigma$ containing $\sigma$. It follows that every $g \in \Sigma_{R} (x) $ belongs to $\Gamma_\sigma$, hence that 
$\Gamma_{R} (x)$ 
is included in $\Gamma_\sigma $, and is thus virtually cyclic too.
\end{proof}

\begin{proof}[End of the proof of Theorem \ref{theo:minordeplacement}] 
Arguing by contradiction, let us suppose that there exists a torsion-free element $\sigma$ of $\Gamma^*$ such that 
$\ell(\sigma) \le  \dfrac{2 (5 D + \delta)}{\nu \big(3^5 e^{73 H (D + 4 \delta)}\big) + 2}$. We then have that  $\ell(\sigma) \le \dfrac{2(2 R - 19 \delta)}{\nu \big(3^5 e^{29 H R}\big) + 2}$, with $R = \frac{5}{2}  (D + 4\delta)$ and, applying Proposition \ref{decrireG3}, there exists a point $x \in X$ such that 
$\Gamma_{R} (x)$ is virtually cyclic for every $R < \frac{5}{2}  (D + 4\delta)$. Hence $\Gamma_{2 D} (x)$ is virtually cyclic and, as $\Gamma_{2 D} (x) = 
\Gamma$ by Proposition \ref{generatorsdiameter}, $\Gamma$ is virtually cyclic, a contradiction.
\end{proof}

\section{Finiteness results for groups}\label{sec:finiteness}

\emph{This section is devoted to the proof of the main theorem. According to Remark \ref{rem:bound} we just need to find a generating set of the $\delta$-hyperbolic groups that we consider whose cardinality is bounded above in terms of the datas. This is the goal of the next subsection.}

\subsection{A bound on the number of generators}

\begin{prop}\label{Nbgenerators}
Given $\delta, H, D > 0$, there exists a constant $N (\delta, H, D)$ with the following property: let $\Gamma$ be a non-cyclic torsion-free group acting properly and by isometries on a $\delta$-hyperbolic space $(X,d)$ verifying $\Ent (X, d) \le H$ 
and $\diam (\Gamma \backslash X) \le D$  then, for every $x \in X$, the generating set
$\Sigma_{10 (D + \delta)} (x)$ of $\Gamma$ has less than $N (\delta, H, D)$ elements.
\end{prop}

The precise value of the constant is
\begin{equation}\label{cteN1}
N (\delta, H, D) := \nu  \left( 3^5 e^{72 H (D+\delta)} \right) \cdot
\left( 1 + \dfrac{20 (D + \delta)}{5 D + \delta}\left( \nu \big(3^5 e^{73 H (D + 4 \delta)}\big) + 2 \right) \right) \, ,
\end{equation}
where $\nu (\cdot)$ is the function defined in Theorem \ref{BGTbasic}.

\begin{proof}[Proof of Proposition \ref{Nbgenerators}] 
The fact that $\Sigma_{10 (D + \delta)} (x)$ is a generating set of $\Gamma$ is a consequence of Proposition \ref{generatorsdiameter}.
By Remark \ref{virtuelcyclique}, $\Gamma$ is non virtually cyclic because it is non cyclic.

Let us choose $R:= 10 (D + \delta)$ and $A = B := \Sigma_{R} (x) = \{\g \in \Gamma : d(x, \g x) \le  R \}$. For every $\eta >0$, Theorem \ref{cocompact2}  implies that
$$ \dfrac{\# (A\cdot B)}{(\# A)^{1/2} (\# B)^{1/2}} = \dfrac{\# (A\cdot A)}{\# A} \le
\dfrac{\mu_{x}^{\Gamma} \big( B_X ( x , 2 (R + \eta))\big)}{\mu_{x}^{\Gamma} \big( B_X ( x , R )\big)} \le 3 \left( \frac{2 (R + \eta)}{R} \right)^{25/4} e^{12 H (\frac{3}{5} R + \eta)} \,.$$
Taking the limit when $\eta$ goes to zero, we deduce that 
$$\dfrac{\# (A\cdot B)}{(\# A)^{1/2} (\# B)^{1/2}} \le 3 \cdot 2^{25/4} e^{12 H (\frac{3}{5} R)}\le 3^5 e^{72 H (D+\delta)}\,.$$

We then apply Theorem \ref{BGTbasic}, with $K=3^5 e^{72 H (D+\delta)}$. We deduce that there exist a virtually nilpotent subgroup $G_0$ and 
a subset $S := \{g_1, \ldots , g_{N}\}$ of $\Gamma$, with $N \leq \nu (K) $ elements, such that 
$\Sigma_{R} (x) \subset \bigcup_{i = 1}^{N} g_i . G_0$ and the cosets $g_i . G_0$ are disjoint. We now define 
$ S':=\{\g_1, \ldots , \g_{n}\}$ to be the subset of those elements $\g \in S$ such that $\g . G_0 \cap \Sigma_{R} (x) \ne \emptyset$. We then have $n \le \nu (K) $, and we obtain that $\Sigma_{R} (x) = \bigcup_{i = 1}^{n} \big(\g_i . G_0 \cap \Sigma_{R} (x) \big)$.
We may furthermore suppose that each $\g_i $ belongs to $\Sigma_{R} (x) $; indeed, if not we can replace $\g_i$ 
by any element of the non empty set $\g_i . G_0 \cap \Sigma_{R} (x)$, that we still denote by $\g_i$.\\
As $\Gamma$ acts co-compactly and properly on the Gromov-hyperbolic space $(X,d)$, the virtually nilpotent subgroup $G_0$ is virtually cyclic by Proposition \ref{actioncocompacte} (v), and hence it is isomorphic to $(\Z , +)$ by Remark \ref{virtuelcyclique}. Then $\g_i . G_0 = \{\g_i  \sigma^k : k \in \Z\}$,
where $\sigma$ is a generator of $G_0$ and the triangle inequality together with the fact that $d(x, \g_i x) \le R$ imply that
$$\# \left(\g_i . G_0 \cap \Sigma_{R} (x)\right)  \le \# \{k : d(x, \g_i \sigma^k x) \le R\} \le \# \{k : d(x, \sigma^k x) \le 2 R\} \le \# \{k : |k| \ell (\sigma) \le 2 R\} \, .$$

Now, by Theorem  \ref{theo:minordeplacement}, $\ell (\sigma) > \dfrac{2 (5 D + \delta)}{\nu \big(3^5 e^{73 H (D + 4 \delta)}\big) + 2}$ since $\Gamma$ is torsion-free. We deduce from the previous inequality that 
$$\# \left(\g_i . G_0 \cap \Sigma_{R} (x) \right) \le \dfrac{4 R}{\ell (\sigma)} + 1 \le 1 + \dfrac{20 (D + \delta)}{5 D + \delta}\left( \nu \big(3^5 e^{73 H (D + 4 \delta)}\big) + 2 \right),$$ 
and consequently that
$$\# \Sigma_{R} (x) = \sum_{i = 1}^{n} \# \left(\g_i . G_0 \cap \Sigma_{R} (x)\right) \le \nu  \left( 3^5 e^{72 H (D+\delta)} \right)
\left( 1 + \dfrac{20 (D + \delta)}{5 D + \delta}\left( \nu \big(3^5 e^{73 H (D + 4 \delta)}\big) + 2 \right) \right).$$
\end{proof}

\subsection{Finiteness for hyperbolic groups}

Given $\delta, H, D > 0$, from the constant $N (\delta, H, D)$ given by \eqref{cteN1}, we define the other constants:
\begin{equation}\label{universalcts}
N_0 (\delta, H, D) := \sum_{i = 0}^{N(\delta , H, D)}  2^{(2 i)^{4\delta + 6}}\,, \quad  N_1 (\delta , H) := N_0 (\delta, H , 1)  \,,
\end{equation}

\begin{equation}\label{universalcts1}
 N_2 (\delta, H, D) :=  N_1 \left(16 \left(\frac{\delta}{D} + 2\right) , 10 H D \right) \, .
\end{equation}

Let us recall that, in a marked group $(\Gamma , \Sigma)$, the generating set $\Sigma$ is supposed to be finite.

\begin{theorem}\label{Nbgroupeshyp} Given $\delta, H > 0$, the set of non cyclic torsion-free $\delta$-hyperbolic marked groups 
$(\Gamma , \Sigma)$ satisfying $\Ent (\Gamma , \Sigma) \le H$ has cardinality (up to isometries) smaller than 
$N_1 : = N_1 ( \delta, H) $.
\end{theorem}

Notice that, by Remark \ref{virtuelcyclique}, when a Gromov-hyperbolic group $\Gamma$ is cyclic and torsion-free, then it is isomorphic to $\Z$.

\begin{proof}[Proof of Theorem \ref{Nbgroupeshyp}]
Let $(\Gamma , \Sigma)$ be a torsion-free and $\delta$-hyperbolic marked group  satisfying $\Ent (\Gamma , \Sigma) \le H$ and let us consider its canonical action on its 
Cayley graph, which is proper and isometric with respect to the algebraic distance $d_\Sigma$ and satisfies $\diam (\Gamma \backslash X) \le 1$.
We may apply Proposition \ref{Nbgenerators}, which implies that the closed $d_\Sigma$-ball $\overline B_\Gamma \big( e , 10 (1 + \delta) \big) = 
\Sigma_{10 (1 + \delta)} (e)$ 
has less that $N (\delta, H, 1)$ elements, where $N (\cdot, \cdot, \cdot)$ is the universal function defined in \eqref{cteN1}. Notice that 
$\Sigma = \overline B_\Gamma \big( e , 1)$, which implies that $\# \Sigma \le N (\delta, H, 1)$.

For the sake of simplicity we shall now write $N$ instead of $N (\delta, H, 1)$. Then, applying Proposition \ref{Nbgroupes} with $k = N $ and 
$N_1 = N_1 ( \delta, H)  :=  \sum_{i = 0}^{N}  2^{(2 i)^{4\delta + 6}}$, yields Theorem \ref{Nbgroupeshyp}.
\end{proof}

\begin{remark}\label{counterfinite}
In Theorem \ref{Nbgroupeshyp}, each hypothesis is necessary, this is proved by detailed examples described in Section \ref{hypnecessary42}.
\end{remark}

\subsection{Finiteness for groups acting on a hyperbolic space}\label{subsec:finitenessaction}

\emph{The next step is to show the finiteness of torsion-free groups acting isometrically, properly and cocompactly on a $\delta$-hyperbolic space. The main issue is that, although these groups are Gromov-hyperbolic, we need to find a generating set for which the hyperbolicity constant of the corresponding marked groups can be computed. Below are precise statements.}

\begin{theorem}\label{boundongroups}
Given $\delta, H, D > 0$, let ${\cal H} (\delta, H, D)$ be the set of torsion-free groups which admit a proper isometric action on some $\delta$-hyperbolic space 
$(X,d)$ satisfying $\Ent (X , d) \le H$ and $\diam (\Gamma \backslash X) \le D$. Then, ${\cal H} (\delta, H, D)$ has cardinality (up to isomorphisms) 
bounded above by the constant $N_2 : = N_2 ( \delta, H, D) < +\infty$ defined in \eqref{universalcts1}.
\end{theorem}

Consider a proper isometric action of a torsion-free group $\Gamma$ on a $\delta$-hyperbolic space $(X,d)$ satisfying $\Ent (X , d) \le H$ and
$\diam (\Gamma \backslash X) \le D$. If $\Gamma$ is cyclic, then it is isomorphic to $\Z$ because it is torsion-free, hence $\Z$ is one element in ${\cal H} (\delta, H, D)$. 
From now on, we only have to consider the case where $\Gamma$ is non cyclic. Let us remark that, since $\Gamma$ is torsion-free, the action, being proper, is faithful and without fixed points; indeed, for $x\in X$, the set of $\g$'s such that $\g x =x$ is a finite group, thus has torsion. 

It comes from \cite{GH} (Th\'eor\`emes 3.22 and 5.12 p. 88) that every group acting properly and cocompactly (by isometries) on a 
Gromov $\delta$-hyperbolic space is a (finitely generated) $\delta'$-hyperbolic group. We need to find a generating set that allows to compute explicitly this number $\delta'$ in terms of $\delta $ and of the diameter of $\Gamma \backslash X$. Theorem 5.12 of \cite{GH} proves that the existence of a $(\lambda , C )$ quasi-isometry from a metric geodesic space $Y$ to a $\delta$-hyperbolic space $X$ implies that $Y$ is $\delta''$-hyperbolic. Following step by step the proof of  this theorem allows to make $\delta''$ explicit in terms of $\delta$, $\lambda$ and $C$. The proof and the computations in \cite{GH} are rather long and involved, mainly because the authors want to take into account the most general case. Here, in Proposition \ref{hyphyp}, we consider a particular case which suits our purpose and 
where the computation of $\delta'$ in terms of $\delta$ is simpler.

\begin{prop}\label{hyphyp}
Let $\Gamma$ be a group which admits an isometric action without fixed points on some $ \delta$-hyperbolic metric space $(X , d)$ such that 
the diameter of $\Gamma \backslash X$ is bounded by $D$; then $\Gamma$ is a $\delta'$-hyperbolic group, where $\delta' = 16 \left(\frac{\delta}{D} + 2\right)$,
when endowed with the algebraic distance associated to the generating set $\Sigma_{10 D} (x)= \{ \g : d(x, \g x) \le 10 D\}$.
\end {prop} 

\begin{proof}[Proof of Proposition \ref{hyphyp}]

We start by two technical lemmas.

\begin{lemma}\label{Gromovgeneral}
Consider a length space $(X,d)$ and a proper isometric action of a group $\Gamma$ on this space such that $\diam (\Gamma \backslash X)\le D$; for any $k \in \N^*$ and any $x \in X$, let $\Sigma := \Sigma_{(k+2) D} (x)= 
\{\g \in \Gamma : d(x, \g x) \le (k+2) D\}$, then $\Sigma $ is a generating set of $\Gamma$ and the corresponding algebraic distance
$d_\Sigma$ verifies
$$ \forall \g , \g' \in \Gamma \ \ \ \ \ \  k D \, \big( d_\Sigma (\g , \g') - 1 \big) \le d(\g x , \g' x) \le (k+2) D \,  d_\Sigma (\g , \g') \, .$$
\end{lemma}

\begin{proof}[Proof of Lemma \ref{Gromovgeneral}] For every $\g \in \Gamma$, we define $ n := \left[ \dfrac{d(x, \g x)}{k D}\right] + 1$. As $(X,d)$ is a length space, there exists a continuous path 
$c$ from $x$ to $\g x$ whose length is $ < n k D$. We choose points $ x = x_0 , x_1 , \ldots , x_n = \g x $ on this path such that 
$d(x_{i-1}, x_i) = k D$ for every integer $1 \le i \le n-1$ and $d(x_{n-1}, x_n) < k D$. This is possible because $\sum_{i = 1}^n d(x_{i-1}, x_i) \le \text{length} (c) < n k D$.

For every $ i \in \{1 , \ldots , n\}$, let $\g_i$ be an element of $\Gamma$ such that $d(x_i , \g_i x ) \le D$, we choose $\g_0 := \id_X$ and $ \g_n := \g$. We define 
$\sigma_i := \g_{i-1}^{-1} \g_i $; one then has
$\g = \g_n = \sigma_1 \, \cdots \sigma_n$, where every $\sigma_i $ belongs to $ \Sigma$, since

$$ d(x , \sigma_i x ) = d( \g_{i-1} x , \g_i x) \le d( \g_{i-1} x , x_{i-1}) + d( x_{i-1} , x_i) +  d( \g_{i} x , x_{i}) \le (k+2) D \,.$$

A first consequence is that $\Sigma$ is a generating set of $\Gamma$, a second one is that
$$d_\Sigma (e, \gamma)  - 1\le n - 1 \le \dfrac{d(x, \g x)}{k D}\,.$$

Let now $\g = s_1 \cdots s_p$ be a decomposition of minimal length of $\g$ as a product of elements of $\Sigma$, the triangle inequality then yields
$$ d(x , \g x) \le \sum_{i = 1}^p d(x, s_i x) \le p (k + 2) D = (k + 2) D\, d_\Sigma (e, \gamma) \, .$$

We thus have $k D \, \big( d_\Sigma (e, \g ) - 1 \big) \le d( x , \g x) \le (k+2) D \,  d_\Sigma (e , \g ) $ and consequently, for every $\g , \g' 
\in \Gamma$, 
$$k D \, \big( d_\Sigma (e, \g^{-1}\g' ) - 1 \big) \le d( x , \g^{-1} \g' x) \le (k+2) D \,  d_\Sigma (e , \g^{-1}\g')\,.$$ 
This ends the proof.
\end{proof}

Let us now suppose that the length space $(X,d)$ is geodesic (see Definitions \ref{geodesicspace}) and that the action of $\Gamma$ has no fixed point.
Let us fix a point $x \in X$.
For each $\sigma \in \Sigma$, we choose one geodesic $[x , \sigma x]$ from $x$ to $\sigma x$, in such a way that the geodesic $[x , \sigma^{-1} x]$ is 
the image by $\sigma^{-1}$ of $[\sigma x , x]$, oriented in the opposite direction. For each $\g \in \Gamma$ and $\sigma \in \Sigma$, we choose 
$\g ([x , \sigma x])$ as the geodesic $[\g x , \g \sigma x]$ from $\g x$ to $\g \sigma x$.
On the Cayley graph ${\cal G} (\Gamma , \Sigma)$ of $\Gamma$ (associated to the generating set $\Sigma$), we consider the distance
which is the natural extension of $d_\Sigma$ such that every edge is isometric to $[0 , 1]$; we still denote this extended distance by $d_\Sigma$ . 
We define the map ${\cal V} : {\cal G} (\Gamma , \Sigma)
\f X$ by setting ${\cal V} (\g) = \g x$ and deciding that, for every $\sigma \in \Sigma$, ${\cal V} $ maps homothetically the edge $[\g , \g \sigma]$ onto 
the geodesic $\g ([x , \sigma x])$.

\begin{lemma}\label{quasiisom}
Under the same assumptions as in Lemma \ref{Gromovgeneral},  if moreover $(X,d)$ is geodesic and if $\Gamma$ acts without fixed points, 
for every pair $ s , t \in {\cal G} (\Gamma , \Sigma)$, one has:\\ 
$  k D \, d_\Sigma (s , t) - (3 k + 2) D \le d \big( {\cal V} (s) , {\cal V} (t) \big) \le (k+2) D \cdot d_\Sigma (s , t) $.
\end{lemma}

\begin{proof}[Proof of Lemma \ref{quasiisom}]
If $s$ and $t$ are on the same edge, the Lemma is trivially verified. Let us now suppose that $s$ and $t$ are not on the same edge:
the point $s$ (resp. $t$) lies on an edge denoted by $[\g_0 , \g_1]$ (resp. by $[g_0 , g_1]$), where the endpoints of these edges are chosen in
such a way that the shortest path from $s$ to $t$ passes through $\g_0$ and $g_0$. One then has simultaneously:
\begin{equation}\label{quasiisom2}
d_\Sigma (s , t)  = d_\Sigma (s , \g_0) + d_\Sigma (\g_0, g_0) + d_\Sigma (g_0, t),\quad
d_\Sigma (s , t) \le  d_\Sigma (s , \g_1) + d_\Sigma (\g_1 , g_1) + d_\Sigma (g_1, t) \,.
\end{equation}
As $s \in [\g_0 , \g_1]$ and as ${\cal V}$ is an homothety from the edge $[\g_0 , \g_1]$, satisfying $d_\Sigma (\g_0 , \g_1) = 1$, of the Cayley graph, 
onto the geodesic $[\g_0 x , \g_1 x ]$, 
with homothety factor $d( \g_0 x , \g_1 x)$, we have $d \big( \g_0 x , {\cal V} (s) \big) = d( \g_0 x , \g_1 x)\cdot d_\Sigma (\g_0, s) \le 
(k+2) D \cdot d_\Sigma (\g_0, s)$. One proves similarly that $d \big( g_0 x , {\cal V} (t) \big)\le (k+2) D \cdot d_\Sigma (g_0, t)$.
Lemma \ref{Gromovgeneral} and the first equation of \eqref{quasiisom2} give:
$$d \big( {\cal V} (s) , {\cal V} (t) \big) \le d \big( \g_0 x , {\cal V} (s) \big) + d(\g_0 x, g_0 x)+ d \big( g_0 x , {\cal V} (t) \big) $$
$$\le (k+2) D \cdot \big(  d_\Sigma (s , \g_0) + d_\Sigma (\g_0, g_0) + d_\Sigma (g_0, t)  \big) =  (k+2) D \cdot d_\Sigma (s , t)  \,,$$
this proves the second inequality of Lemma \ref{quasiisom}.

By addition of the two estimates \eqref{quasiisom2}, using the equality $ d_\Sigma (\g_0 , s) + d_\Sigma (s , \g_1) = 1$ and its analogous for $t$, we get:
\begin{equation}\label{quasiisom1}
k D \,d_\Sigma (s , t)  \le \frac{1}{2}\, k D \,\big( d_\Sigma (\g_0, g_0) + d_\Sigma (\g_1, g_1) + 2 \big) \le 
\frac{1}{2}\big( d (\g_0 x , g_0 x)  +  d (\g_1 x , g_1 x) \big) + 2 k D ,
\end{equation}

where the last inequality follows from Lemma \ref{Gromovgeneral}.
On the other hand, we have
$$\begin{array}{l} d \big( {\cal V} (s) , {\cal V} (t) \big) \ge d (\g_0 x , g_0 x) - d \big( {\cal V} (s) , \g_0 x\big) - d \big( {\cal V} (t) , g_0 x\big)\,,\\
d \big( {\cal V} (s) , {\cal V} (t) \big) \ge d (\g_1 x , g_1 x) - d \big( {\cal V} (s) , \g_1 x\big) - d \big( {\cal V} (t) , g_1 x\big)\, .
\end{array}$$
and, by addition of these two inequalities, using the estimate 
$$ d \big(  \g_0 x ,  {\cal V} (s) \big) + d \big( {\cal V} (s) , \g_1 x\big) =  d \big( \g_0  x , \g_1 x\big) \le (k+ 2) D$$

 (because $ \g_0^{-1} \g_1 \in \Sigma$) and the analogous estimate for $ d \big(  g_0 x ,  {\cal V} (t) \big) 
+ d \big( {\cal V} (t) , g_1 x\big)$, we get:
$$ d \big( {\cal V} (s) , {\cal V} (t) \big)  \ge \frac{1}{2} \big( d (\g_0 x , g_0 x)  +  d (\g_1 x , g_1 x) \big) -  (k+ 2) D
\ge k D \,d_\Sigma (s , t) - (3 k + 2) D \, ,$$
where the last inequality follows from \eqref{quasiisom1}. This proves the first inequality of Lemma \ref{quasiisom}.
\end{proof}

For any two points $ s , \, t$ of the Cayley graph ${\cal G} (\Gamma , \Sigma)$, denote by $[s,t]$ a shortest path joining $s$ to $t$ in 
${\cal G} (\Gamma , \Sigma)$, {\sl i.e.} a geodesic of ${\cal G} (\Gamma , \Sigma)$ endowed with the accessibility distance $d_\Sigma$; then 
${\cal V} ([s,t])$ is a piecewise-geodesic path from ${\cal V} (s)$ to ${\cal V} (t)$ in $X$ which satisfies:
$$\text{Length}\big( {\cal V} ([s,t])\big ) \le (k+2) D\, d_\Sigma (s,t) \le \frac{k + 2}{k}\, d \big( {\cal V} (s) , {\cal V} (t) \big) + \frac{(k + 2)(3 k + 2)}{k} D
\, ,$$
where the first inequality follows, by addition of geodesic lengths, from the second inequality of Lemma \ref{quasiisom} and where the second
inequality comes from the first inequality of Lemma \ref{quasiisom}.

Choosing $k := 8$ in the sequel then, by Lemma \ref{Gromovgeneral},
$\Sigma := \Sigma_{10 D} (x)$ is a generating set of $\Gamma$ and the last inequality becomes:
\begin{equation}\label{dvsdsigma}
\text{Length}\big({\cal V} ([s,t])\big) \le \frac{5}{4}\, d \big( {\cal V} (s) , {\cal V} (t) \big) + \frac{65}{2}\, D\, .
\end{equation}
Now, taking also $k = 8$ in  Lemma \ref{quasiisom}, we obtain:
\begin{equation}\label{dvsdsigma1}
 d \big( {\cal V} (s) , {\cal V} (t) \big) \ge 8 D \, d_\Sigma (s , t) - 26\, D \,.
\end{equation}
Inequations \eqref{dvsdsigma} and \eqref{dvsdsigma1} respectively prove that the hypotheses (ii) and (i) of Proposition \ref{quasiisometric} are verified
with the following values of the parameters: $ a = 8 D$, $b = 26 D$, $\lambda = \frac{5}{4}$ and $C = \frac{65}{2} D$;
Proposition \ref{quasiisometric} then implies that the group $\Gamma$, endowed with the distance $d_\Sigma$ associated to the generating set $\Sigma := \Sigma_{10 D} (x)$ is $\delta'$- hyperbolic with 
$$\delta' = \dfrac{4}{a} \,\left( (6 \lambda^2\ + 14 \lambda + 5) \delta  + \frac{4 \lambda + 3}{6 \lambda + 2} \,C + b\right) \le 16 \left( \frac{\delta}{D} + 2 \right) \, .$$
\end{proof}

\begin{proof}[Proof of Theorem \ref{boundongroups}]
Now, one has, by a classical argument, 
$$ \Ent \big(\Gamma, \Sigma_{10 D} (x) \big) \le 10 D \Ent (X,d) \le 10  H D\,.$$ 
We may then apply Theorem \ref{Nbgroupeshyp}, where we replace $\delta$ and $H$ respectively by $16 \left(\frac{\delta}{D} + 2\right)$ and $10 H D$, this shows that the set of marked groups $\big(\Gamma, \Sigma_{10 D} (x) \big) $ has cardinality, up to isometries, bounded by $N_2 (\delta, H , D)$.
\end{proof}

\begin{remark}
In Theorem \ref{boundongroups}, each hypothesis is necessary, this is proved by detailed examples described in Section \ref{hypnecessary44}.
\end{remark}

\section{Finiteness results for  spaces}\label{sec:topo}

\emph{We now derive several consequences of the previous finiteness theorems.}

\subsection{Finiteness for $\delta$-hyperbolic spaces up to quasi-isometries}\label{subsec:quasiisometric}

Let us start by recalling some standard definitions.

\begin{defi}\label{Space} Given $\lambda \ge 1$ and $c , a \ge 0$, two metric spaces $(X,d_X)$ and $(Y,d_Y)$ are said to be $(\lambda, c , a)$-quasi isometric 
if there exist $f : X \f Y$ and $h : Y \f X$ satisfying the following properties for every $x, x' \in X$ and every $y,y' \in Y$:
\begin{itemize}
\item $d_Y \big(f (x) , f (x') \big) \le \lambda \, d_X (x,x') + c$ and $d_X \big(h (y) , h (y') \big) \le \lambda \, d_Y (y,y') + c$,
\item $d_X \big(x \,, \, h \circ f (x) \big) \le a$ and $d_Y \big(y \,, \, f \circ h (y) \big) \le a$
\end{itemize}
\end{defi}
Up to trivial changes in the parameters, this definition is equivalent to the other classical ones.

\begin{defi} 
For the sake of simplicity, we shall say that a family $\mathcal F$ of metric spaces has cardinality smaller than $q < + \infty$ up to  $(\lambda, c , a)$-quasi isometries
whenever there exist metric spaces $(X_1, d_1), \ldots , (X_q, d_q)$ in $\mathcal F$ such that any other metric space $(X , d)\in \mathcal F$ is $(\lambda, c , a)$-quasi isometric to some of the $(X_i, d_i)$'s.
\end{defi}

\begin{defi}\label{Space} Given $\delta , H , D > 0$ we denote by $\mathcal M_{hyp} (\delta, H , D)$ the set of $\delta$-hyperbolic spaces, with entropy 
$\le H$, which admit the action of a non-cyclic, torsion-free, uniform lattice $\Gamma\subset\mathrm{Isom}(X, d)$ such that the diameter of the quotient $\Gamma\backslash X$ is $\le D$. For the sake of simplicity $\Gamma$ is called a $D$-lattice for $(X,d)$.
\end{defi}

Recalling that $N_2 ( \delta, H, D) $ is the constant defined in \eqref{universalcts1}, we have the

\begin{theorem}\label{theo:quasiisometric}
The set $\mathcal M_{hyp} (\delta, H , D)$ has cardinality bounded above by $N_2(\delta, H, D)$ up to  $ \left(\frac{5}{4} , \frac{25}{2}\, D,  D \right)$-quasi isometries.
\end{theorem}
\begin{proof}
A triple $(X,d, x)$ will be called a \emph{pointed 
element of} $\mathcal M_{hyp} (\delta, H , D)$ if
$(X,d) \in \mathcal M_{hyp} (\delta, H , D)$ 
and if $x$ is a point of $X$.
Given any pointed element $(X,d, x)$ of $\mathcal M_{hyp} (\delta, H , D)$, 
a marked group $(\Gamma_X, \Sigma_X)$ is called a \emph{pointed $D$-lattice of 
$(X,d, x)$} if $\Gamma_X$ is a $D$-lattice for $ (X,d)$ and if $\Sigma_X$ is the generating set $\Sigma_{10 D} (x) := \{\g \in \Gamma_X \setminus \{1\} : 
d (x , \g x) \le 10 D\}$ of $\Gamma_X$ (see Proposition \ref{generatorsdiameter}). We are thus focusing on the orbit of $x$ under the action of $\Gamma_X\subset\mathrm{Isom}(X, d)$. In this case, the action of $\Gamma_X$ on $(X, d)$ is proper (by definition of a lattice) and fixed point-free by Lemma \ref{autofidele} (iv). Proposition \ref{hyphyp} then asserts that the marked group $(\Gamma_X, \Sigma_X)$ is $\delta'$-hyperbolic, where $\delta' = 16 \left(\frac{\delta}{D} + 2\right)$. By a classical argument, we also have that
$\Ent (\Gamma_X, \Sigma_X) \le 10 D \Ent (X,d) \le 10  H D$. Theorem \ref{Nbgroupeshyp} then ensures that the set of pointed $D$-lattices of pointed elements of $\mathcal M_{hyp} (\delta, H , D)$ has cardinality bounded above by $N_2 ( \delta, H, D) $, up to isometries. More precisely any pointed $D$-lattice $(\Gamma_X, \Sigma_X)$ of a pointed element $(X, d, x)$ of $\mathcal M_{hyp} (\delta, H , D)$ is isometric to some pointed $D$-lattice $(\Gamma_{X_0}, \Sigma_{X_0})$ in a pointed element $(X_0, d_0, x_0)$ taken in a list of cardinality $\le N_2 ( \delta, H, D) $. Let us denote by $\varrho$ this isometric isomorphism.\\
We shall now show that $(X,d)$ is then $\left( \frac{5}{4} , \frac{25}{2}\, D, D \right)$-quasi isometric to $(X_0 , d_0)$. Indeed, for every $z \in X$, let us denote by $\g_z$ one of the elements of $\Gamma_X$ such that $d \big(z, \g_z .x \big) = \Min_{\g \in \Gamma_X}\, d(z, \g .x)$. Similarly, for every 
$y \in X_0$, let us denote by $g_y$ one of the elements of $\Gamma_{X_0}$ such that $d_0 \big(y, g_y .x_0 \big) = \Min_{g \in \Gamma_{X_0}}\, d_0(y, g. x_0)$.
By definition of a $D$-lattice, we have:

\begin{equation}\label{Dapprox} 
d \big(z, \g_z .x \big) \le D \ \  \text{ and } \ \ d_0 \big(y, g_y .x_0\big) \le D
\end{equation}

Let us define $f : X\f X_0$ by $f(z) = \varrho (\g_z).x_0$ and $h : X_0 \f X$ by $h(y) = \varrho^{-1} (g_y).x$. Using twice Lemma \ref{Gromovgeneral} 
(with $k = 8$) and the isometry $\varrho$, we get:
$$\frac{1}{10 D}\,d_0 \big(\varrho (\g_z).x_0 , \varrho (\g_{z'}).x_0\big) \le  d_{\Sigma_{X_0}} \big(\varrho (\g_z) , \varrho (\g_{z'})\big) 
=  d_{\Sigma_{X}} \big(\g_z , \g_{z'}\big) \le \frac{1}{8 D}\, d(\g_z.x, \g_{z'}.x) + 1 \, .$$
Plugging in this last inequality the estimate $ d(\g_z.x, \g_{z'}.x) \le d(z,z')+ 2 D$, which follows from \eqref{Dapprox}, we deduce:
$$d_0 \big( f(z) , f(z') \big) \le \frac{5}{4} \, d(z,z') + \frac{25}{2}\, D \ \ \text{ and similarly }\ \ d \big( h(y), h(y' \big) \le \frac{5}{4}\, d(y , y') + \frac{25}{2}\, D \, .$$
Let us denote $f(z)=\varrho (\g_z) .x_0$ by $u$ for the sake of simplicity. As the action is fixed point-free, $\varrho (\g_z) $ is the unique element of 
$\Gamma_{X_0}$ where $g \mapsto d_0 \big( u , g.x_0 \big)$ attains its minimum. A consequence is that $g_u = \varrho (\g_z) $ and that 
$h \circ f (z) = h (u) = \varrho^{-1} (g_u).x = \g_z . x$. Using \eqref{Dapprox}, it follows that $d \big(z, h \circ f (z) \big)= d(z , \g_z.x) \le D$ for every
$z \in X$. Analogous arguments prove that $d_0 \big(y, f \circ h (y) \big)= d_0 (y , g_y.x_0) \le D$. Hence $f$ and $g$ provide a $ \left(\frac{5}{4} , \frac{25}{2}\, D,  D \right)$-quasi isometry between $(X,d)$ and $(X_0, d_0)$.
\end{proof}

\subsection{Homotopical finiteness}

We recall the following characterisation of the existence of the universal covering:

\begin{prop}\label{Spanier} \emph{(\cite{span} Corollaries 2.5.14 and 2.5.15)}
A connected, locally path-connected topological space $X$ has a (simply connected) universal covering if and only if it is semi-locally  simply connected
(for a definition of this notion, see \cite{span}).
\end{prop}

Given a semi-locally simply connected length space $(X, d)$, it is automatically connected and path-connected, hence it admits a universal 
covering $\pi : \widetilde X \f  X$. We moreover can define the length of a path $c : [0 , a] \f \widetilde X$ as the length (w.r.t. the distance $d$) of the path 
$\pi \circ c : [0 , a] \f  X$ in $(X ,d)$ and define the pull-back distance $\tilde d$ on $\widetilde X$ as the associated length distance.

\begin{defis}\label{boundedness} 
Let ${\cal M}$ be any family of topological spaces and $q$ a given integer.
\begin{itemize}
\item We say that  ${\cal M}$ contains at most $q$ fundamental groups  if there exist $ X_1, \ldots , X_q  \in {\cal M}$ such 
that every $X \in {\cal M}$ has a fundamental group isomorphic to the fundamental group of one of the $X_i$'s.
\item We say that ${\cal M}$ contains at most $q$ homotopy types (resp. at most $q$ topologies) if there exist $ X_1, \ldots , X_q  \in {\cal M}$
such that every $X \in {\cal M}$ is homotopically equivalent (resp. is homeomorphic) to one of the $X_i$'s.
\end{itemize}
\end{defis}

In order to define families of metrisable spaces which verify such boundedness conditions, we previously recall some classical definitions. The first 
of these notions is the following one, introduced by K. Borsuk:

\begin{defi}
A metrisable topological space $X$ is said to be an ANR if, for every other metrisable topological space $Y$ and any embedding $\iota : X \hookrightarrow Y$ such that $\iota (X)$ is closed in $Y$, there exists a neighbourhood of $\iota (X)$ in $Y$ which retracts onto $\iota (X)$.
\end{defi}
Notice that a metrisable space with finite topological dimension is an ANR if and only if it is locally contractible (see 
\cite{Hu}, Theorem V.7.1 ), in particular every locally finite CW-complex is an ANR.

\begin{defi}
A topological space $X$ is said to be {\em aspherical} if $\pi_i (M)$ is trivial for every $i >1$.
\end{defi}

Let us now define the families of metric spaces which verify the boundedness properties introduced in Definitions \ref{boundedness}

\begin{defis}  Given $\delta, H , D > 0$, we now consider classes of compact connected topological spaces $X$
which admit a compatible length distance $d$ with diameter $\le D$ such that the universal cover of $(X,d)$, is $\delta$-hyperbolic and 
has entropy $\le H$. Notice that the universal cover of $(X,d)$ exists in the three cases below,

\begin{itemize}
\item ${\cal M} (\delta, H , D)$ is the set of such $X$ which are semi-locally simply connected, with torsion-free 
fundamental groups.

\item ${\cal M}^* _{\rm anr}(\delta, H , D)$ is the set of such $X$ which are aspherical ANR spaces.

\item ${\cal M}_{\rm man} ^*(\delta, H , D)$ is the set of such $X$ which are closed aspherical topological manifolds of dimension $\ne 4$.
\end{itemize}
\end{defis}
The following lemma gives the inclusions between these sets.
\begin{lemma}\label{inclusions} 
For every $\delta , H , D > 0$, one has ${\cal M}_{\rm man} ^*(\delta, H , D) \subset {\cal M}^* _{\rm anr}(\delta, H , D) \subset {\cal M} (\delta, H , D)$.
\end{lemma}

\begin{proof}
The inclusion ${\cal M}_{\rm man} ^* (\delta, H , D) \subset {\cal M}^* _{\rm anr}(\delta, H , D) $ is an immediate consequence of the fact that every topological manifold
is an ANR. Notice that every topological manifold has a universal cover.

Let us now  prove that ${\cal M}^* _{\rm anr} (\delta, H , D) \subset  {\cal M} (\delta, H , D)$. 
Every $X \in {\cal M}^* _{\rm anr} (\delta, H , D)$, being an ANR, is 
locally contractible, thus it is automatically semi-locally simply connected. Moreover, being a compact ANR,  $X$ is homotopically equivalent to a 
finite CW-complex $X'$ (see \cite{Wes} at page 119), whose dimension is thus finite. Moreover, as $X$ is aspherical, $X'$ is an aspherical CW-complex.
Hence the cohomological dimension (see definition page 185 of \cite{Bwn}) of the fundamental group $\Gamma_{X'}$ of $X'$ is finite, for it is bounded 
from above by the dimension of the aspherical CW-complex $X'$ (see \cite{Bwn}, chapter VIII, Proposition 2.2, page 185). Using Corollary 2.5, in 
chapter VIII, page 187 of \cite{Bwn}, it comes that $\Gamma_{X'}$ is torsion-free. As the fundamental group $\Gamma_{X}$ of $X$ is 
isomorphic to the fundamental group $\Gamma_{X'}$ of $X'$, it follows that $\Gamma_X$ is torsion-free and this ends the proof.
\end{proof}

Finally let us recall that $N = N (\delta, H , D)$ is the constant defined in \eqref{cteN1} and define 

\begin{equation}\label{constant-homotop}
N'_0 = N'_0 (\delta, H , D) := 1 + \sum_{k = 0}^{N}   2^{(2 k)^3}.
\end{equation}

 An hyperbolic space being supposed geodesic and proper, we have the

\begin{theorem}\label{boundedpi1s}
The set ${\cal M} (\delta, H , D)$  contains at most $ N'_0 (\delta, H , D)$ fundamental groups (up to isomorphisms) and the set ${\cal M}^* _{\rm anr}(\delta, H , D)$  contains at most $ N'_0 (\delta, H , D)$ homotopy classes.
\end{theorem}

The first assertion of Theorem \ref{boundedpi1s} could be obtained as a corollary of Theorem \ref{boundongroups},
with an upper bound $N_2 (\delta, H , D) >> N'_0 (\delta, H , D)$ for the number of fundamental groups. However, as the proof of Theorem \ref{boundongroups} is rather involved, it is worth giving a simplified  and quantitatively improved version here using the fact that the groups under consideration act on simply connected spaces. 

\begin{proof}[Proof of Theorem \ref{boundedpi1s}]
For every $X\in {\cal M} (\delta, H , D)$, the fundamental group of $X$ being 
torsion-free, it is cyclic if and only if it is isomorphic to $\Z$. We thus define ${\cal M}_0 $ as the set of those elements of ${\cal M} (\delta, H , D)$ whose fundamental group is isomorphic to $\Z$ and we may now only consider spaces lying in ${\cal M} (\delta, H , D) \setminus {\cal M}_0 $, which hence have non cyclic 
fundamental groups.
For every $X \in {\cal M} (\delta, H , D) \setminus {\cal M}_0 $, let $d$ be a compatible metric such that $(X,d)$ is a length space with diameter $\le D$ whose universal cover $(\widetilde X, \tilde d)$ is $\delta$-hyperbolic with entropy bounded above by $H$.
It is well known that, since $X$ is assumed to be Hausdorff and proper, then the canonical action of $\Gamma_X =\pi_1(X)$ on $\widetilde X$ is properly discontinuous ({\sl i.e.} the set $\{ g \in \pi_1(X) \, | \, gU \cap V \neq 0 \}$ is finite, for every open neighbourhoods $U,V$ of $x$, see for instance \cite{BoN}, Chapter III, section 4.4, Proposition 7)
and  isometric.
Applying Proposition \ref{Nbgenerators} to this action, 
we obtain that, for every $\tilde x \in \widetilde X$, $\Sigma_{10 (D + \delta)} (\tilde x)$
has less that $N (\delta, H, D)$ elements, where $N (\cdot, \cdot, \cdot)$ is the function defined in \eqref{cteN1}. This yields
$\# \big(\Sigma_{2 D} (\tilde x) \big) \le N (\delta, H, D)$. Now Corollary \ref{Serre1} proves that there exists a finite list $\Gamma_1 , \ldots , 
\Gamma_{N'_0-1}$ of groups such that the fundamental group $\Gamma_X$ of $X$ is isomorphic to some of the $\Gamma_i$'s. This proves the first assertion of Theorem \ref{boundedpi1s}.

Now, as ${\cal M}^* _{\rm anr}(\delta, H , D) \subset  {\cal M} (\delta, H , D)$ by Lemma \ref{inclusions}, we can apply this first assertion of Theorem \ref{boundedpi1s} which 
proves that we can choose  $X_1 , \ldots , X_{N'_0} \in {\cal M}^* _{\rm anr}(\delta, H , D) $ such that,
for every $X \in {\cal M}^*_{\rm anr} (\delta, H , D)$, there exists $i \le N'_0$ and an isomorphism $\varrho$ from the fundamental group 
of $X$ to that of $X_i $. As $X_i $ is aspherical and connected, there exists a continuous map $f : X \f X_i$ such that the induced map between the fundamental groups of
$X$ and $X_i$ is $f_* = \varrho$. Moreover, as both spaces $X$ and $ X_i $ are path-connected and aspherical, $f$ induces 
an isomorphism 
$\pi_k (X) \f \pi_k (X_i)$ for every $k \in \N$. Hence $X$ and $ X_i $ are homotopically equivalent by the ANR version of Whitehead's Theorem (see \cite{BD}, Theorem 2.5).
\end{proof}

\subsection{Topological finiteness}\label{topfini}

In order to deduce topological finiteness from homotopical finiteness, we recall the

\begin{defi}
A closed topological manifold $M$ is said to be {\em topologically rigid} if every homotopy equivalence with another closed topological manifold $N$ is homotopic 
to a homeomorphism. 
\end{defi}
Deciding in which cases a closed aspherical manifold is topologically rigid is known as \lq \lq solving the Borel conjecture" in these cases.

We recall that $N'_0 = N'_0 (\delta, H , D)$ is given in Theorem \ref{boundedpi1s}. We now define 
$$N''_0 (\delta, H , D) := 2 N'_0 (\delta, H , D) \cdot 2^{1 + e^{400 H D}}\quad\textrm{and}\quad N'''_0 (\delta, H , D) := 2 N''_0 (\delta, H , D). $$
The main result of this subsection is the following corollary of Theorem \ref{boundedpi1s}.

\begin{theorem}\label{boundedtopology}
Given $\delta,\, H,\, D$, the set ${\cal M}_{\rm man}^* (\delta, H , D)$ contains at most $N'''_0 (\delta, H , D)$ topologies.
\end{theorem}
In dimension $n\ne 4$, the finiteness ``up to homeomorphisms'' given by Theorems \ref{boundedtopology} and \ref{3manifolds}
can be promoted to finiteness ``up to diffeomorphisms'' by classical results of Kirby-Siebenmann [Ki-Si] and of Hirsch-Mazur [Hi-Ma] on PL-structures and their smoothings.  
Moreover, finiteness up to homeomorphisms can be proved also in dimension $4$, with the extra assumption that the manifolds  under consideration  are nonpositively curved (or more generally if their universal covering is a Busemann space); we will not pursue this matter here, see  the answer to Question 3.54 in \cite{BG2}, and \cite{BCGS2} for improvements and complete proofs.

We shall first prove a version of Theorem \ref{boundedtopology}, in dimension $3$, where we drop the  asphericity assumption. 

Indeed, let ${\cal M}^n_{\rm man} (\delta, H , D)$ be the set of closed
topological $n$-manifolds $X$, with torsion-free fundamental group, which admit a compatible length distance $d$ with diameter $\le D$ such that the universal cover of $(X,d)$ is $\delta$-hyperbolic  and has entropy $\le H$. We then get the following result:

\begin{theorem}\label{3manifolds}
Given $\delta, H , D > 0$, the set  ${\cal M}^3_{\rm man} (\delta, H , D)$ contains at most $N''_0 (\delta, H , D)$ topologies.
\end{theorem}

\begin{proof}[Proof of Theorem \ref{3manifolds}]
By Theorem \ref{boundedpi1s}, there exist groups $\Gamma_1 , \cdots , \Gamma_{p}$ where $p \le   N'_0 
(\delta, H , D)$ such that the fundamental group of each $ X \in {\cal M}^3 _{\rm man}(\delta, H , D)$ is isomorphic to one of the $\Gamma_i$'s. 
We decompose each $\Gamma_i$ as a free product  $\Gamma_i =  \mathbf{ Z}   \ast  \cdots  \ast  \mathbf{ Z}  \ast \Gamma_i (1) \ast   \cdots  \Gamma_i (q_i)$, with each $\Gamma_i(k) \not\simeq  \mathbf{ Z}$   and  not further decomposable.
Let now $X$ be any closed $3$-manifold  whose fundamental group is isomorphic to $\Gamma_i$.
By the proof of Kneser's Conjecture, to the above decomposition of $\Gamma_i$ in free products corresponds a decomposition of $X$ in a connected sum
$Z(1) \# \ldots \# Z(h_i) \# X(1) \# \ldots \# X(q_i)$ where, to each factor of $\Gamma_i $ which is isomorphic to  $\mathbf{ Z}$, corresponds a prime factor $Z(h)$ of $X$ which is a $S^2$-bundle over $S^1$, and to each  $\Gamma_i (k)$  corresponds an irreducible factor $X(k)$ of $X$ (see \cite{hempel}, Lemma 3.13). The torsion-free assumption guarantees that none of the factors $X(k)$ is a lens space.

As a consequence of the Sphere Theorem (see Theorem 4.3 of \cite{hempel} and \cite{luck}), an irreducible factor is aspherical if and only if its fundamental group is infinite and without elements of order 2. As our manifolds  have torsion-free fundamental group, the irreducible factors $X(k)$ are then aspherical; hence, by the solution of the Borel conjecture in dimension 3 (see Remark after Theorem 2.1.2 in \cite{AFW},   and  \cite{heil}), the factors $X(k)$ are determined up to homeomorphisms by the groups  $\Gamma_i (k)$. 
On the other hand the factors $Z(h)$ are either homeomorphic to $S^2 \times S^1$, or to the non-orientable twisted $S^2$-bundle over $S^1$ denoted by $S^2 \widetilde{\times} S^1$.
In the orientable case,  the homeomorphism class of $X$ is determined by the orientations of its prime factors, and   each $Z(h)=S^2 \times S^1$. Therefore the number of orientable manifolds in ${\cal M}^3 _{\rm man} (\delta, H , D)$ whose fundamental group is   isomorphic to 
$\Gamma_i$ is bounded above by $2^{n_i}$, where $n_i$ is the number of its prime  factors.
In the non-orientable case, the homeomorphism class  of $X$ only depends on the homeomorphism type of its  prime factors. However each $Z(h)$ is homeomorphic to $S^2 \times S^1$ or to $S^2 \widetilde{\times} S^1$,  and again the number of non-orientable manifolds in ${\cal M}^3 _{\rm man} (\delta, H , D)$ whose fundamental group is   isomorphic to 
$\Gamma_i$ is bounded above by $2^{n_i}$.
 It follows that the  number of topologies in ${\cal M}^3 _{\rm man}(\delta, H , D)$ is finite, bounded by  
\begin{equation}\label{nbtopologies1}
 2 \cdot \left( 2^{n_1} + \ldots + 2^{n_{p}} \right)  < +\infty \ .
\end{equation}
We recall that the rank of a group is the minimal cardinality of a generating set, note that  $n_i \le  \text{\rm rank} (\Gamma_i)$. This, and the fact that a free product acts $0$-acylindrically on its Bass-Serre 
tree, allows to apply Inequality (2) of \cite{CS3}, which implies that $$ n_i \le \text{\rm rank} (\Gamma_i)  \le 1 + e^{200 \Ent (\Gamma_i , \Sigma_{2 D} (x))} \le  1 + e^{400 \, H D} $$
where, in the last two inequalities, we used Proposition \ref{generatorsdiameter}. Plugging this estimate in \eqref{nbtopologies1} ends the proof.
\end{proof}

\begin{proof}[End of the proof of Theorem \ref{boundedtopology}]
In dimension 2, the theorem follows from Theorem \ref{boundedpi1s} and from the classification of compact surfaces. In dimension 3, it follows from Theorem \ref{3manifolds}. Therefore, 
we can suppose that all the manifolds under consideration have dimension $\ge 5$.
As the fundamental group $\Gamma$ of any manifold $X \in {\cal M}^\ast_{\rm man} (\delta, H , D) $ is Gromov-hyperbolic, because it acts cocompactly on 
the Gromov-hyperbolic universal cover of $X$, we can apply the following theorem of A.~Bartels and W.~L\"uck (\cite{BL}, Theorem A): \emph{every closed aspherical topological manifold of dimension $\ge 5$, whose fundamental group is hyperbolic, is topologically rigid.} 
This implies that all the manifolds $X \in {\cal M}^\ast_{\rm man} (\delta, H , D) $ of dimension $ \ge 5$ are topologically rigid. 
Now, as ${\cal M}^\ast_{\rm man} (\delta, H , D) \subset {\cal M}^\ast_{\rm anr} (\delta, H , D) $ (see Lemma \ref{inclusions}), by the second assertion of  Theorem \ref{boundedpi1s}, we can choose $X_1 , \ldots , X_{N'_0} \in {\cal M}^\ast_{\rm man} (\delta, H , D) $ such that every $X \in {\cal M}^\ast_{\rm man} (\delta, H , D)$ is homotopically equivalent to one of these $X_i$. 
 As $X$ is topologically  rigid, in dimension $\ge 5$, this proves that $X$ is homeomorphic to $X_i$.
\end{proof}

\begin{remark}\label{Kpi1nec}\emph{About the asphericity assumption in Theorems \ref{boundedpi1s}  and  \ref{boundedtopology}}.

\noindent By the classification of closed surfaces and by Theorem \ref{3manifolds}, this assumption can be dropped for surfaces and 3-manifolds whose fundamental group is torsion-free.
In contrast, asphericity is necessary in dimension  $n\ge 4$: indeed, for given values of $\delta,\, H,\, D$, in each dimension $n\ge 4$, Example 3.70 of \cite{BG2} exhibits infinitely many non-aspherical, 2 by 2 non-homotopically 
equivalent manifolds belonging to ${\cal M}_{\rm man} ^n (\delta, H , D)$.
 \end{remark}

\section{Examples and counter-examples}\label{sec:exa}

\emph{In this section we show that all the hypotheses of Theorems \ref{Nbgroupeshyp} and \ref{boundongroups} are necessary. For this purpose we consider each hypothesis, one by one, and exhibit  appropriate examples and/or counter-examples.}

\subsection{On the hypotheses of Theorem \ref{Nbgroupeshyp}}\label{hypnecessary42}

\emph{We start the justification for the assumptions of Theorem \ref{Nbgroupeshyp}.}

\subsubsection{The torsion-free assumption}\label{subsubsec:torsion1}

It is a necessary assumption. Indeed, counter-examples are constructed as follows: we first choose a sequence $(G_i, S_i)_{i \in \N}$ of finite, 2 by 2 non isomorphic, marked groups whose Cayley graphs all have diameters bounded from above by the same value\footnote{For example, every finite group admits a 
generating set such that the corresponding Cayley graph has diameter $1$: it is sufficient to take the whole group as a generating set.} $D$. We now choose a 
non cyclic torsion-free $\delta_0$-hyperbolic marked group $(\Gamma_0 , \Sigma_0)$ satisfying $\Ent (\Gamma_0 , \Sigma_0) \le H_0$. For each $i \in \N^*$, 
consider the direct product of marked groups $(\Gamma_i , \Sigma_i)$, where $\Gamma_i := \Gamma_0 \times G_i $ and $\Sigma_i := \Sigma_0 \times \{1\}
\cup \{ 1\} \times S_i$. As $ d_{\Sigma_0} (1 , \g) \le d_{\Sigma_i} \big( (1,1) , (\g , g) \big) \le d_{\Sigma_0} (1 , \g)  + D$, we have 
$ \Ent (\Gamma_i , d_{\Sigma_i} ) =  \Ent (\Gamma_0 , d_{\Sigma_0}) \le H_0$ and $ (\Gamma_i , d_{\Sigma_i})$ is $\delta$-hyperbolic, with 
$\delta := 2 (\delta_0 + D)$. This yields an infinite sequence of groups $\Gamma_i$ verifying  the hypotheses of Theorem \ref{Nbgroupeshyp}, except the 
\lq \lq torsion-free" assumption. Hence, if one forgets the hypothesis \lq \lq torsion-free", there are infinitely many classes of groups (modulo isomorphisms) which verify the other hypotheses of Theorem \ref{Nbgroupeshyp}.

\subsubsection{The non-cyclicity}

Let us consider the sequence of marked groups $(\Z , S_n)_{n \in \N^*}$, where the generating set is $S_n := \{-n, \ldots , -1 ,\, 1 , \ldots , n\}$. 
They form an infinite family of marked groups, two by two non isometric, though they satisfy all the hypotheses of Theorem \ref{Nbgroupeshyp}, except that they all are cyclic.

It is obvious that $\Z$ is torsion-free and satisfies, $\forall n \in \N^*$, $\Ent (\Z , S_n) = 0$ because the word distances verify $d_{S_n} \ge \frac {1}{n} d_{S_1}$.
Moreover, each group $(\Z , S_n)$ is $\delta$-hyperbolic with $\delta = 12$ as follow from the following computations. 

For every pair of points $x , y  \in \Z$, considered as vertices of the Cayley graphs of both $(\Z , S_1)$ and of $(\Z , S_n)$, one has 
\begin{equation}\label{ds1infdsn0}
 d_{S_1} (x,y) \le n\, d_{S_n} (x,y) < d_{S_1}(x,y) + n \, .
\end{equation}
Let us now consider any quadruple of points $x, y , u , v$ of the Cayley graph of $(\Z , S_n)$, choose $x' $ (resp. $y'$) as the first (resp. the last) vertex crossed by a geodesic of this graph joining $x$ and $y$ and choose $u' $ (resp. $v'$) as the first (resp. the last) vertex crossed by a geodesic of this graph 
joining $u$ and $v$. 
Using first the second inequality \eqref{ds1infdsn0}, then Lemma \ref{proprietes} (iii) and the fact that the Cayley graph associated to $(\Z , S_1)$ is isometric 
to the real line (hence is $0$-hyperbolic), we obtain

$$ d_{S_n}(x,y) + d_{S_n} (u,v)  \le d_{S_n}(x',y') + d_{S_n} (u',v') + 4 \le \frac{1}{n}\,\big( d_{S_1} (x',y') + d_{S_1} (u',v'))\big) + 6 $$
$$\le\frac{1}{n}\, \Max \big(d_{S_1}(x',u') + d_{S_1} (y',v') \, ,\, d_{S_1}(x',v') + d_{S_1} (u',y') \big) + 6 \, .$$
Using the first inequality \eqref{ds1infdsn0}, it follows that
$$ d_{S_n}(x,y) + d_{S_n} (u,v)  \le  \Max \big(d_{S_n}(x',u') + d_{S_n} (y',v') \, ,\, d_{S_n}(x',v') + d_{S_n} (u',y') \big) + 6 $$
$$ \le \Max \big(d_{S_n}(x,u) + d_{S_n} (y,v) \, ,\, d_{S_n}(x,v) + d_{S_n} (u,y) \big) + 6 \, .$$
From this last inequality, using Propositions 1.1.6 (p.3) and 1.3.1 (p.7) of \cite{CDP}, we deduce that $(Z, S_n)$ is $12$-hyperbolic for every $n \in \N^*$ 
(we recall that $(\Z , S_1)$ is $0$-hyperbolic).

\subsubsection{Boundedness of the entropy}\label{subsubsec:entropy1}

 Indeed, let $( \mathbf F_i , \mathbf S_i)$ be the free group with $i$ generators, endowed with its canonical symmetric generating set $\mathbf S_i$ (with $\# \mathbf S_i = 2 i$). It is torsion-free, non cyclic 
(when 
$i \in \N \setminus \{ 0 , 1\}$), $0$-hyperbolic (thus $1$-hyperbolic), but one has $\Ent ( \mathbf F_i , \mathbf S_i) = \ln (2 i-1)$. Hence, when $i$ runs in $\N \setminus \{ 0 , 1\}$, we get an infinite family of marked groups, 2 by 2 non isomorphic, which verify all the hypotheses of Theorem \ref{Nbgroupeshyp} except the boundedness of the entropy.

\subsubsection{Boundedness of the hyperbolicity constant}

Gromov-hyperbolicity is obviously a necessary condition. Indeed, for every integer $n \ge 2$, $\Z^n$, endowed with its canonical generating set, is 
torsion-free, non cyclic, with zero entropy, but it is not Gromov-hyperbolic. It is an infinite sequence, even up to isometries, in contradiction with the conclusion 
of Theorem \ref{Nbgroupeshyp}.

Even if the marked groups under consideration verify all the other hypotheses of Theorem \ref{Nbgroupeshyp}, and if they
are only supposed to be Gromov-hyperbolic, without a uniform bound of their hyperbolicity constants, one does not even have a finite number 
of marked groups up to isometry.

We shall construct below infinitely many generating sets of $\Z \ast \Z$ such that the corresponding marked groups 
are two by two non isometric though they all verify the hypotheses of Theorem \ref{Nbgroupeshyp}, except for the fact that the hypothesis 
\lq \lq $\delta$-hyperbolic" is replaced by \lq \lq Gromov-hyperbolic".
Given any pair $p , q$ of integers which are mutually prime, we define $S_0 := \{-1 , 1\}$ and $S_{p,q} := \{p, q, -p, -q\}$. The 
Bachet-B\'ezout theorem proves that $S_{p,q}$ is a generating set for 
$\Z$. A consequence is that a generating set for the free product $\Z \ast \Z$ is given by the disjoint union $S_0 \bigsqcup S_{p,q}$ of the subset 
$ S_0 \simeq S_0 \ast \{0\}$ of the first factor and of the subset $S_{p,q} \simeq \{0\}\ast S_{p,q}$ of the second factor of $\Z \ast \Z$. 
We then have the following properties:

\begin{itemize}
\item The marked groups $(\Z \ast \Z , S_0 \bigsqcup S_{p,q})$ are all Gromov-hyperbolic (with unbounded hyperbolicity constants) and verify all 
the other hypotheses of Theorem \ref{Nbgroupeshyp}.

Indeed, $\Z \ast \Z $ is non cyclic and torsion-free and the entropy of $(\Z \ast \Z , S_0 \bigsqcup S_{p,q})$ is bounded from above by the entropy of the 
free group with $6$ generators, {\sl i.e.} by $\ln (11)$. Moreover, as $\Z \ast \Z$, endowed with the canonical generating set $S_0 \bigsqcup S_{0}$ is the free group 
with $2$ generators, it is $0$-hyperbolic, thus $(\Z \ast \Z , S_0 \bigsqcup S_{p,q})$, being quasi-isometric to $(\Z \ast \Z , S_0 \bigsqcup S_{0})$, is 
Gromov-hyperbolic. 

\item Furthermore, if $S_{p,q} \ne S_{p',q'}$, then the two marked groups $(\Z \ast \Z , S_0 \bigsqcup S_{p,q})$ and $(\Z \ast \Z , S_0 \bigsqcup S_{p',q'})$ are not isometric, and this implies that the cardinality (up to isometries) of the family of marked groups 
$\{(\Z \ast \Z , S_0 \bigsqcup S_{p,q})\}_{(p,q) \in \Z^2}$ is infinite.
The proof of this fact goes by contradiction. Let us suppose that there exists an isometry
$\phi : (\Z \ast \Z , S_0 \bigsqcup S_{p,q}) \f (\Z \ast \Z , S_0 \bigsqcup S_{p',q'})$. We then denote by $\Z_1$ and $\Z_2$ respectively the first and second 
factors $\Z\ast \{0\}$ and $\{0\} \ast \Z$ in the free product $\Z \ast \Z$, and by $G_1$ and $G_2$ the images of $\Z_1$ and $\Z_2$ by $\phi$. We also denote by 
$F$ the image by $\phi$ of $\{0\}\ast S_{p,q}$, it is a generating set for the cyclic group $G_2$, because $\phi$ is an isomorphism. Now $F\subset S_0 \bigsqcup S_{p',q'}$, 
since $\phi$ is an isometry; then $F$ is included in $\{0\}\ast S_{p',q'}$; if not $F$ would intersect
both sets $S_0 \ast \{0\}$ and $\{0\}\ast S_{p',q'}$, hence it would contain a pair $\{a , b\}$ of non trivial elements such that $a \in \Z_1$ and $b \in \Z_2$.
Consequently the cyclic subgroup $G_2$ would contain the free group generated by $\{a , b\}$: a contradiction. This implies that $F = \{0\}\ast S_{p',q'}$.

Since $\phi$ maps $\{0\}\ast S_{p,q}$ onto $\{0\}\ast S_{p',q'}$ and $\Z_2$ onto $\Z_2$, its restriction $\varphi$ to the second factor is an isometry between the marked groups 
$(\Z , S_{p,q})$ and $(\Z , S_{p',q'})$. This implies that $\varphi (1) = \pm 1$, thus that $\varphi = \pm \id$ and that $ S_{p',q'} = \varphi (S_{p,q}) = S_{p,q}$.
\end{itemize}
Along the way, Theorem \ref{Nbgroupeshyp} also proves that the hyperbolicity constant must go to infinity when $p$ and $q$ go to infinity.

\subsection{On the hypotheses of Theorem \ref{boundongroups}}\label{hypnecessary44}

\emph{The necessity of the torsion-free assumption and of the boundedness of the entropy are already clear from the counterexamples given in \ref{subsubsec:torsion1} and in  \ref{subsubsec:entropy1} }

\subsubsection{Boundedness of the diameter}
Let us consider again the example \ref{subsubsec:entropy1}, where 
we replace the metric $d_{\mathbf S_i}$ of the Cayley graph $X_i$ by $d_i = i \cdot d_{\mathbf S_i}$. From previously, we know that $(X_i, d_i)$ is $0$-hyperbolic 
and verifies 
$$\Ent (X_i, d_i) = \frac{1}{i} \Ent (\mathbf F_i, S_i) = \frac{\ln (2 i - 1)}{i} \le 1\,.$$
Hence, we get a family $(\mathbf F_i)_{i \in \N^*}$ of torsion-free groups, whose action on $(X_i, d_i)$ verifies all the hypotheses of 
Theorem \ref{boundongroups} except the boundedness of the diameter. The conclusion of Theorem \ref{boundongroups} is not verified since the $\mathbf F_i ,$'s are infinitely many and  2 by 2 non isomorphic.
 
\subsubsection{Boundedness of the hyperbolicity constant}

Let ${\cal H} (H, D)$ be the set of torsion-free groups which admit a proper isometric action on some Gromov-hyperbolic space $(X,d)$ with entropy 
bounded by $H$ and $\diam (\Gamma \backslash X) \le D$. The groups belonging to ${\cal H} (H, D)$ verify all the hypotheses of Theorem \ref{boundongroups}, except that the spaces on which they act are only supposed to be Gromov-hyperbolic, without any uniform bound on their hyperbolicity constants, the conclusion of Theorem \ref{boundongroups} is then not valid for ${\cal H} (H, D)$, {\sl i.e.} ${\cal H} (H, D)$ contains infinitely many 2 by 2 non isomorphic groups. More precisely, we construct a sequence of 2 by 2 non isomorphic groups belonging to the set ${\cal H} (H, D)$ (but not to ${\cal H} (\delta, H, D)$) as follows. These groups are the fundamental groups of closed hyperbolic $3$-manifolds obtained by a sequence of Dehn surgeries from a given finite volume hyperbolic manifold with one cusp. We will consider their actions on their Cayley graphs associated to an appropriate choice of a generating set.

We consider a non compact $3$-dimensional hyperbolic manifold $(X,g)$ with finite volume and, for the sake of simplicity, only one cusp, 
denoted by $C$. Let $\pi : 
(\mathbf H^3 , can ) \f (X,g)$ be the universal covering of $(X,g)$. We shall first consider the fundamental group $\Gamma$ of $(X,g)$ as the group of 
automorphisms of this covering. 
Provided that $C$ is chosen deep enough in  $X$, we choose one connected component $K$ of $\pi^{-1} (C)$. It is an horoball whose boundary $\partial K$ is an horosphere, hence $K$ is diffeomorphic
to $\partial K \times \R^+$.  Any $\g \in \Gamma$ exchanges the connected components of $\pi^{-1} (C)$ and, denoting by $P$ the parabolic 
subgroup of those elements $\g$ verifying $\g (K) = K$, it is classical that the cusp $C$ is the quotient of $K$ by $P$ and that the \emph{cusp torus} $T = \partial C$ is the quotient of $\partial K$ by $P$.

Let us now fix a basis $\{ \alpha , \beta \}$ of $H_1(P\backslash\partial K ,  \Z) = H_1(T, \Z) \simeq \Z^2$, we then have $P \simeq \Z^2$.
Given any pair of mutually prime integers $(a, b)$, one defines the \emph{Dehn filling with slope $(a,b)$} as the closed manifold $X_{a,b}$ constructed as follows: 
we glue a solid torus $\mathbf T_0 = D^2 \times \mathbf S^1$ to $X \setminus C$, identifying their boundaries via a diffeomorphism $\phi : \partial\mathbf T_0 = \partial D^2 
\times \mathbf S^1 \f T$ which maps $\partial D^2$ onto a simple\footnote{On the torus $\mathbf T^2$, endowed with the flat metric such that
$\{\alpha , \beta\}$ is orthonormal, the proof of the existence of a simple geodesic representative of the homology class $a \alpha +b \beta$ is immediate 
when $a,b$ are mutually prime.} 
closed curve $c(a,b)$ which represents the homology class $a\alpha +b\beta$ and the factor $ \mathbf S^1$ onto a representative $c(a',b')$ of any homology class 
$a' \alpha +b' \beta$ such that $a b' - b a' = \pm 1$. It is now classical that the differentiable structure of $X_{a,b}$ does not depend on the choice of 
$(a',b')$ (see for example \cite{Mart}, Lemma 10.1.2), nor on the choice of the representative $c(a,b)$ of the homology class $a\alpha +b\beta$ 
(see \cite{Mart}, Lemma 10.1.2 and Proposition 6.3.16). 
By the Van Kampen theorem, the fundamental group of $X_{a,b}$ (denoted by $\Gamma_{a,b}$) is isomorphic to $\Gamma/ \ll c(a,b) \gg=\Gamma_{a,b}$, where 
$\ll c(a,b) \gg$ is the smallest normal subgroup of $\Gamma$ which contains the homotopy class of the loop $c(a,b)$ (see for instance 
\cite{Mart}, Proposition 10.1.3.). Let us now fix a generating set $S$ of $\Gamma$, define $N'_1 := \# (S)$. Denoting by $p_{a,b}$ the quotient map 
$\Gamma \f \Gamma_{a,b}$, then $S_{a,b} := p_{a,b} (S)$ is a generating set for $\Gamma_{a,b}$ which verifies 
$\# (S_{a,b}) \le N'_1$.

We shall now consider the proper, isometric, canonical action of the group $\Gamma_{a,b}$ on its Cayley graph $G_{a,b}$ associated to the generating set $S_{a,b}$ and endowed with the word distance $d_{a,b} := d_{S_{a,b}}$. It is trivial that $\diam \big(\Gamma_{a,b} \backslash G_{a,b}\big) \le 1 $.
On the one hand, the entropy of $(G_{a,b} , d_{a,b})$ is bounded from above by the 
entropy of the free group generated by $S_{a,b}$, hence $\Ent (G_{a,b} , d_{a,b}) \le \ln (2 n -1)$, where $n= \#S_{a,b}$.

On the other hand,  it is known, since the works of W. Thurston (\cite{Thurston}), that there exists a finite subset 
$E$ of $\Z^2$ (the set of \emph{exceptional slopes}) such that $X_{a,b}$ admits an hyperbolic Riemannian metric, denoted by $g_{a,b}$, for every coprime couple $(a,b) \notin E$. Let us denote by $I$ the infinite set of all coprime $(a,b) \in \Z^2 \setminus E$, we shall suppose in the sequel that $(a,b) \in I$. 
As $\Gamma_{a,b}$ is the fundamental group of the hyperbolic manifold $(X_{a,b} , g_{a,b})$, acting co-compactly on its universal cover, the group 
$\Gamma_{a,b}$ is torsion-free and Gromov-hyperbolic, hence its Cayley graph $(G_{a,b} , d_{a,b})$ is Gromov-hyperbolic too.
This proves that $\Gamma_{a,b} \in {\cal H} (H, D)$.\\ 
Now, a result of T. Jorgensen, revisited by W. Thurston (\cite{Thurston}) (for a complete explanation, see \cite{NZ}, particularly equation (1)) implies that
$$\forall (a,b) \in I, \quad \Vol ( (X_{a, b} , g_{a, b})) < \Vol (X , g) \quad \text{\rm and} \quad \lim_{(a,b) \to \infty} \Vol (X_{a, b} , g_{a, b}) = \Vol (X, g) \,. $$
Hence, there exists a sequence $(a_i,  b_i)_{i \in \N}$ of elements of $ I$, going to infinity, such that 
$$\forall i \in \N, \quad  \Vol (X_{a_i, b_i},  g_{a_i, b_i})  < \Vol (X_{a_{i+1}, b_{i+1}},  g_{a_{i+1}, b_{i+1}})\,. $$
It follows that, for $i\ne j$, 
$(X_{a_i, b_i}, g_{a_i, b_i})$ and $(X_{a_j, b_j}, g_{a_j, b_j}) $ are not isometric thus, by Mostow's rigidity Theorem, that their fundamental
groups $\Gamma_{a_i, b_i}$ and $\Gamma_{a_j, b_j}$ are not isomorphic. This proves that $(\Gamma_{a_i, b_i})_{i \in \N}$ is a sequence of two by two non isomorphic elements of ${\cal H} (H, D)$.

\section{Appendices}\label{appendices}

\subsection{Geodesic and length spaces}\label{geodesics}

A result by M. Gromov (\cite{Gr1} Proposition 3.22, whose proof, written for Riemannian manifolds, is still valid on path-connected metric spaces
is the following

\begin{prop}\label{generatorsdiameter}
For every group $\Gamma$ which admits a proper isometric action on a path-connected metric space $(X,d)$ such that $\diam (\Gamma \backslash X) \le D$, and for every 
$x \in X$, the subset $\Sigma_{2 D} (x) := \{ \sigma \in \Gamma^*  : d(x , \sigma x) \le 2 D\}$ is a symmetric generating set of $\Gamma $.
\end{prop}

\begin{proof}
By the path-connectedness, for every $\g \in \Gamma$ and every $\e > 0$, there exists a finite set $ \{y_0, y_1 , \ldots , y_N\} \subset X$ verifying 
$y_0 = x$, $ y_N = \g x$ and $d(y_{i-1}, y_i) < \e $ for every $i \in \{1 , \ldots , N \}$. Let us choose $\g_0, \g_1 , \ldots , \g_N \in \Gamma$ such that $ \g_0 = e$, $\g_N = \g $ and 
$d (y_i , \g_i x) \le D$, we then get $\g = \sigma_1 \cdot \ldots  \cdot \sigma_N$, where 
$\sigma_i = \g_{i-1}^{-1}\cdot \g_i \in \Sigma_{2 D+ \e} (x)$ and the finiteness of $\Sigma_{3 D} (x)$ 
proves that $\Sigma_{2 D+ \e}   (x) = \Sigma_{2 D} (x)$ when $\e$ is sufficiently small. This proves that $\Sigma_{2 D} (x) = \Sigma_{2 D + \e} (x)$ 
is a symmetric generating set of $\Gamma $, the properness of the action implying the finiteness of $\Sigma_{2 D} (x) $. 
\end{proof}

The following definitions are classical (see for example \cite{BH} D\'efinitions I.1.3 p. 4):
\begin{defis}\label{geodesicspace}
In any metric space $ (X,d)$
\begin{itemize}
\item  a \emph{(normal) geodesic} is a map $c$ from some interval $I \subset \R$ to $X$ such that, for every $ t, \, t' \in I$ 
$ d(c(t) , c(t') ) = |t - t'|$,

\item when $I$ is a closed interval (resp. $ ]-\infty , + \infty [$) the geodesic is called a \emph{geodesic segment}
(resp. a \emph{geodesic line}), the image of a geodesic segment $ c $ with origin $x$ and endpoint $y$ is often denoted by $[x, y]$ 
(though this does not suppose that this geodesic segment is unique),

\item a \emph{(normal) local geodesic}  is a map $c$ from some interval $I \subset \R$ to $X$ such that, for every $ t \in I$, 
there exists $\e > 0$ such that $ d(c(t') , c(t'') ) = |t' - t''|$ for every $ t' , \, t'' \in \, ]t- \e , t+ \e[\,\cap I$,

\item a metric space $(X,d)$ is \emph{geodesic} if any two points can be joined by at least one geodesic.
\end{itemize}
\end{defis}

\begin{defi}\label{naturalparameter}
Given a geodesic segment $[x_0 , x_1]$, the \emph{natural parametrization} of this segment is the map $t \mapsto x_t$ from $[0 , 1]$ to $[x_0 , x_1]$, 
defined by $d(x_0 , x_t) = t \, d(x_0 , x_1)$.
\end{defi}

\subsection{About Gromov-hyperbolic spaces}\label{GromovHyp}

\subsubsection{Definitions}
Given three nonnegative numbers $\alpha , \beta, \gamma$, we define the tripod $T := T(\alpha , \beta, \gamma)$ as the metric simplicial tree 
with $3$ vertices
$x', \,y' ,\, z'$ of valence $1$ (the \emph{endpoints}), one vertex $c$ of valence $3$ (the \emph{branching point}), and $3$ edges  $[cx'] , \, [cy'] , \, [cz']$ 
of respective lengths $\alpha , \beta, \gamma$ (the \emph{branches}). We denote by $d_T (u,v) $ the distance on this tree between two points $u,v \in 
T$, {\sl  i.e.}  the minimal length of a path contained in $T$ and joining $u$ to $v$.

\smallskip
For the sake of simplicity, we only consider geodesic metric spaces (see Definition in section \ref{geodesics}). In such a space a geodesic triangle 
$\Delta = [x , y , z]$ is the union of three geodesics $[x , y ]$, $[y , z ]$ and $[z , x ]$. Given three points $x,y,z$ in a geodesic metric space, there
exists at least one geodesic triangle $\Delta = [x , y , z]$ whose sides have respective lengths $d(x,y),\, d(y,z)$ and $d(x,z)$. 

\begin{lemma}\label{prodist} 
To any geodesic triangle $\Delta$ corresponds a metric tripod $(T_\Delta , d_T)$ and a surjective map $f_\Delta : \Delta \f  T_\Delta$ (called the 
\emph{approximation of $\Delta$ by a tripod}) such that, in restriction to each side of $\Delta$, $f_\Delta$ is an isometry,
\end{lemma}

Indeed, if $\Delta = [x, y, z]$, $T_\Delta$ is constructed as the tripod $ T(\alpha , \beta, \gamma)$, where (by the triangle inequality) $(\alpha , \beta, \gamma)$ is the 
unique element of $[0 , +\infty[^3$ such that $d(x, y) = \alpha + \beta$, $d(x, z) = \alpha + \gamma$ and $d(y, z) = \beta + \gamma$. This choice of 
$(\alpha , \beta, \gamma)$ implies the existence of the map $f_\Delta : \Delta \f  T_\Delta$ as asserted in Lemma \ref{prodist}.

\begin{defis}\label{hypdefinition0}
A geodesic triangle $\Delta$ of $(X,d)$ is said to be \emph{$\delta$-thin} if, for every $u \in T$ and every $v,  w \in f_\Delta^{-1} (\{u\})$, one has $d(v,w) \le \delta$.

In this text, a metric space is said to be \emph{$\delta$-hyperbolic} if it is geodesic, proper, and if all its geodesic triangles are $\delta$-thin.
\end{defis}

The following results are well known (and often taken as definitions of $\delta$-hyperbolicity); their proof may be found (with variable constants, depending 
on the choice of the definition) in any classical source (see for example \cite{GH}, \cite{CDP}, \cite{BH}).

\begin{lemma}\label{proprietes} For every $\delta$-hyperbolic space $(X,d)$, one has:

\begin{itemize}

\item[(i)]  \emph{(Tripod Approximation)} For every geodesic triangle $\Delta $, its approximation $f_\Delta : \Delta \to (T_\Delta, d_T)$ by a tripod verifies
 $$d(u,v) - \delta  \le d_T \big(f_\Delta (u) , f_\Delta (v) \big) \le d(u,v).$$

\item[(ii)] \emph{(Rips Lemma)} Every geodesic triangle $\Delta = [x , y , z]$ is \emph{$\delta$-slim}, {\sl i.e.} for every point $ u \in [y,z]$, one has: 
$d(u,[x,y ]\cup [x,z ]) \le \delta$. Reciprocally, every $\delta$-slim triangle is $4 \delta$-thin \emph{(see \cite{GH}, Proposition 2.21 and proof at 
page 43)}.

\item[(iii)] (\emph{Quadrangle Lemma})  For every four points $  x, y , z, w \in X$, one has
$$d(x,z) + d(y,w) \le \Max \big( d(x,y) + d(z,w) \, ;\, d(x,w) + d(y,z) \big) + 2\,\delta \, .$$
\end{itemize}
\end{lemma}

The proof of the properties (i) and (ii) of the following Lemma is given in \cite{GH}, Proposition 2.5 p. 45.

\begin{lemma}\label{convexity} \emph{(Quasi-Convexity)}
In every $\delta$-hyperbolic space $(X,d)$, let $[x_0 , x_1]$ and $[y_0 , y_1]$ be two geodesic segments endowed with their natural 
parametrizations $t \mapsto x_t$ and $t \mapsto y_t$ (see Definition \ref{naturalparameter}) , then
\begin{itemize}
\item[(i)] if $x_0 = y_0$, then $d(x_t, y_t) \le t \, d(x_1 , y_1) + \delta$;
\item[(ii)] in the general case $d(x_t, y_t) \le (1-t) d(x_0 , y_0) + t \, d(x_1 , y_1) + 2 \,\delta$.
\end{itemize}
\end{lemma}

\subsubsection{Projections and almost equality in the triangle inequality}\label{sectionprojections}

The proofs of the following results can be found (for example) in \cite{BCGS}, section 8.2.

\begin{defi}\label{projete}
For every closed subset $F$ of a metric space $(X,d)$ and every point $x \in X$, a \emph{projection of $x$ on $F$} is a (any) point $\bar x \in F$ such that
$ d(x,\bar x) = d(x,F) := \inf_{z \in F} d(x,z)$.
\end{defi}
When it exists, a projection of $x$ on $F$ is generally not unique.
By continuity of the distance there always exists a projection of $x$ on $F$ when $F$ is compact, and when $F$ is closed if the metric space is proper. 
However, the assumption \emph{proper space} is not necessary when $F$ is the image of a geodesic, as proved by the following
\begin{lemma} \emph{(Projection on Geodesic)} \label{projectiongeod}
If $c$ is a geodesic line (or ray, or segment) in a metric space $(X,d)$, every point $x \in X$ admits a projection on the image ${\rm Im}(c)$ of $c$
and the map $ x \mapsto d(x , {\rm Im}(c) )$ is Lipschitz with Lipschitz constant $1$.
\end{lemma}

\begin{lemma}\label{projection} \emph{(almost equality in the triangle inequality)}
In a $\delta$-hyperbolic space $(X,d)$, for a geodesic $c$ of $(X,d)$ and a point $y$ on it and for every point $x \in X$, any of its projections
$\bar x$ on the geodesic $c$ verifies $ d(x,y) \ge d(x,\bar x) + d(\bar x , y) - 2\, \delta$.
\end{lemma}

\begin{lemma}\label{ecartement}
In a $\delta$-hyperbolic space $(X,d)$, for every geodesic (segment or line) $c$, for any pair of points $x,y \in X$, if $\bar x$ and $\bar y$ are projections 
of $x$ and $y$ (respectively) on $c$, then one has:
$$ d(\bar x , \bar y) > 3\, \delta  \implies  d(x,y) \ge d(x,\bar y) + d( \bar y , y) - 4 \delta \ge d(x,\bar x) + d(\bar x , \bar y) + d( \bar y , y) - 6\, \delta \ .$$
\end{lemma}

\begin{proof}
By Lemma \ref{proprietes} (iii), we have:
\begin{equation}\label{quadrilatere}
d(x, \bar y) + d(y , \bar x) \le \Max \big( d(x,y) + d(\bar x , \bar y) \,;\,  d(\bar x , x) +  d(\bar y , y)\big) + 2 \delta \,  .
\end{equation}
From  Lemma \ref{projection}, we get 
$$d(x, \bar y) + d(y , \bar x) \ge d(x, \bar x)  + 2 \, d(\bar x , \bar y) + d(y , \bar y) - 4 \delta > d(x, \bar x) + d(y , \bar y) + 2 \delta\,.$$ 
Plugging this estimate in \eqref{quadrilatere} gives $d(x, \bar y) + d(y , \bar x) \le d(x,y) + d(\bar x , \bar y) +
 2 \delta$ and, using Lemma \ref{projection}, that 
$$ d(x, \bar x)  + 2 \, d(\bar x , \bar y) + d(y , \bar y) - 4 \delta \le d(x, \bar y) + d(y , \bar y) +  d(\bar x , \bar y) - 2\delta \le d(x,y) + d(\bar x , \bar y) + 2 \delta \, ,$$
and this proves Lemma \ref{ecartement}.
\end{proof}

For more informations about Gromov-hyperbolic metric spaces see the original publication \cite{Gr4} and \cite{GH}, \cite{CDP} and \cite{BH}.

\subsection{Quasi-isometry and hyperbolicity}

We are going to prove a weak version of Theorem 5.12 of \cite{GH}, which is sufficient for our purposes and whose proof, under our 
stronger hypotheses, is easier to follow and to make explicit. In Theorem 5.12 of \cite{GH} it is proved that the existence of a $(\lambda , C )$ 
quasi-isometry from a metric geodesic space $Y$ to a $\delta$-hyperbolic space $X$ implies that $Y$ is $\delta''$-hyperbolic. The value of
$\delta''$ becomes explicit if one follows the computations contained in pages 82 to 88 of  \cite{GH}. Our weak version of Theorem 5.12 of \cite{GH} is the	

\begin{prop}\label{quasiisometric} Let $(Y,d_Y)$ and $(X, d_X)$ be two proper and geodesic metric spaces, if there exists a continuous map $f : Y \f X$ satisfying the following properties
\begin{itemize}
\item[(i)] for every $y , y' \in Y$, $d_X \big(f(y) , f(y') \big) \ge a\cdot d_Y(y,y') - b $,
\item[(ii)] for every pair of points $y, y' \in Y$ and every geodesic $[ y , y']$ connecting these two points, $f([y,y'])$ is rectifiable and 
$\text{\rm Length} \big(f ([y , y'])\big) \le \lambda \cdot d _Y\big( f (y) , f (y')\big) + C$ 
\end{itemize}
then, if $(X, d_X)$ is $\delta$- hyperbolic, $(Y,d_Y)$ is $\delta''$- hyperbolic with 
$$\delta'' = \dfrac{4}{a} \,\left( (6 \lambda^2\ + 14 \lambda + 5) \delta  + \frac{4 \lambda + 3}{6 \lambda + 2} \,C + b\right) \le 
\dfrac{4}{a} \,\left( (6 \lambda^2\ + 14 \lambda + 5) \delta  + C + b\right).$$
\end {prop} 

Theorem 5.12 of  \cite{GH} is stronger than this proposition, because it only assumes that $f$ is a  quasi-isometry, {\sl i.e.} that the 
path $f ([y , y'])$ is a quasi-geodesic which may be discontinuous.

\medskip
The first step in the proof of Proposition \ref{quasiisometric} is Proposition 1.6 of Chapter III of \cite{BH}, which may be written

\begin{lemma}\label{quasigeodesic1}
In a $ \delta$-hyperbolic metric space $(X , d)$, for every continuous rectifiable path $ c : [0 , 1] \f X$, every point $x$ of any geodesic segment 
joining $c(0)$ to $c(1)$ verifies 
$$d(x, \text{\rm Im} (c) ) \le \delta \left( 1 + \log_2^+  \left(\frac{\text{\rm length} (c)}{\delta}\right)\right), \quad 
\text{\rm where} \quad \log_2^+ (x) := \max \big(\log_2 (x) , 0  \big).$$
\end {lemma} 

The second step in the proof of Proposition \ref{quasiisometric} is the following result, which can be found in \cite{BH} (chapter III, 
Th\'eor\`eme 1.7), we only explicit here the constants $C_1$ and $C_2$ and simplify the proof\footnote{Indeed, in \cite{BH} (chapter III), 
the proof of Theorem 1.7 is longer because the authors have to take into account the case of non continuous quasi-geodesics, this is 
the aim of their Lemma 1.11 (p. 403) to prove that the solution in the non continuous case is a consequence of the solution in the continuous case.}.

\begin{prop}\label{quasigeodesic}
In a $ \delta$-hyperbolic metric space $(X , d)$, given $\lambda \ge 1$ and $C \ge 0$, for every continuous rectifiable path $ c : [0 , a] \f X$
which verifies $\text{\rm Length} \big(c ([t , t'])\big) \le \lambda \cdot d \big( c(t) , c(t')\big) + C$, for every $[t , t'] \subset [0 , a]$,
\begin{itemize}
\item[(i)] every geodesic segment $[c(0) , c(a)]$ lies in the $C_1$-neighbourhood of the image of $c$, where 
$C_1 := (6 \lambda + 2) \delta + \dfrac{C}{6 \lambda + 2}$
\item[(ii)] the image of $c$ lies in the $C_2 $--neighbourhood of the geodesic segment $[c(0) , c(a)]$, where $C_2 = C_2 (\lambda , C, \delta) 
:= (1 + \lambda)\, C_1 + \frac{C}{2}$.
\end{itemize}
\end {prop}

\begin{proof}[End of the proof of Proposition \ref{quasiisometric}]
Let $\Delta := [y_0 , y_1 , y_2]$ be a geodesic triangle in $(Y,d_Y)$, {\sl i.e.} the union of geodesics $[y_0 , y_1 ]$, $[ y_1 , y_2 ]$ and
$[ y_2 , y_0 ]$ of $(Y,d_Y)$. Let $u$ be any point of the side $[ y_1 , y_2 ]$. By the hypothesis (ii), the $\delta$-hyperbolicity of $(X,d)$ and 
Proposition \ref{quasigeodesic} (ii), one has $d_X \big( f(u), [f(y_1), f(y_2)]\big) < C_2$, for any choice of the geodesic segment 
$[f(y_1), f(y_2)]$ joining the endpoints of the path $f ([ y_1 , y_2 ])$. We now choose a point $ v \in [f(y_1), f(y_2)]$ such that $d_X \big( f(u), v \big) 
< C_2$. Let us choose geodesics $[f(y_0), f(y_1)]$ and $[f(y_0), f(y_2)]$ in order to complete the geodesic triangle $\Delta' :=
[f(y_0), f(y_1) , f(y_2)]$; as the triangle $\Delta'$ is $\delta$-slim (see Lemma \ref{proprietes} (iii)), there exists $v' \in [f (y_0) , f (y_1)] \cup 
[f (y_0) , f (y_2)] $ such that $d_X(v,v') \le \delta$.

By what we just have proved, one has either $v' \in [f (y_0) , f (y_1)] $, and then, by the hypothesis (ii), the $\delta$-hyperbolicity of $(X,d)$ 
and Proposition \ref{quasigeodesic} (i), one has $d_X \big( v'\, ,\, f ([y_0 , y_1]) \big)  < C_1 $, or $v' \in [f (y_0) , f (y_2)] $, and then, by the same argument, 
one has $d_X \big( v'\, ,\, f ([y_0 , y_2]) \big)  < C_1 $; in both cases, there exists 
$u' \in [y_0 , y_1] \cup [y_0 , y_2]$ such that $d_X \big(v', f (u') \big) < C_1 $. Hence, for every $ u \in [ y_1 , y_2 ]$ there exists
$u' \in [y_0 , y_1] \cup [y_0 , y_2]$ such that $d_X \big(f(u),f(u')\big) <  C_1 + C_2 + \delta$, thus (using the hypothesis (i)) such that 
$d_Y (u , u' ) <  \dfrac{1}{a}\, (C_1 + C_2 + b +\delta)\ $. This proves that the triangle $\Delta$ is $\frac{\delta''}{4}$-slim (in the sense of Lemma 
\ref{proprietes} (ii)) where $ \delta'' = \dfrac{4}{a} \,(C_1 + C_2 + b + \delta) $, thus (using  Lemma \ref{proprietes} (ii)) that it is $\delta''$-thin 
(in the sense of Definitions \ref{hypdefinition0}). We conclude that $(Y, d_Y)$ is $\delta''$-hyperbolic.
\end{proof}

\subsection{About actions of groups}\label{aboutactions}

\subsubsection{General properties}

Let $\Gamma$ be a group acting by isometries on a metric space. We recall that $\Gamma$ is said to be discrete if it is a discrete subgroup of the isometry group of $(X,d)$ for the compact-open topology. For every $R> 0$, and every $x \in X$, we defined
$\Sigma_R (x) := \{\g \in \Gamma : \g\,x \in \overline B_X(x, R)\}$.

The following Proposition is proved (for example) in \cite{BCGS} Proposition 8.12:

\begin{prop}\label{discret1}
On a metric space $(X,d)$ every faithful and proper action by isometries is discrete.
Conversely, if $(X,d)$ is a proper metric space, then every faithful and discrete action by isometries is proper.
\end{prop}

\begin{lemma}\label{autofidele}
Every proper action by isometries of a group $\Gamma$ on a metric space $(X,d)$ verifies:
\begin{itemize}
\item[(i)] the quotient space $\Gamma\backslash X$ is a metric space when endowed with the quotient-distance $\bar d$ defined by
$ \bar d (\Gamma \cdot x, \Gamma \cdot y) :=  \inf_{\g \in \Gamma} d( x , \g \, y) $,

\item[(ii)] if $(X,d)$ is a proper space and if $ (\Gamma\backslash X , \bar d)$ has finite diameter, then $\Gamma\backslash X$ is compact,

\item[(iii)] if $\Gamma$ acts without fixed points, then the action is faithful and discrete.

\item[(iv)] if $\Gamma$ is torsion-free, then the action is faithful, discrete and without fixed points
\end{itemize}
\end{lemma}
For a proof of this result see (for example) \cite{BCGS}, Lemma 8.13.

The following remark is well known, for a proof, see for instance \cite{BCGS}, Remark 8.14

\begin{remark}\label{virtuelcyclique}
Every non trivial, torsion-free, virtually cyclic group is isomorphic to $(\Z , +)$.
\end{remark}

\subsubsection{Properties of actions on Gromov-hyperbolic spaces}

Let $(X , d)$  be any $\delta$-hyperbolic space, denote by $\partial X $ its \emph{ideal boundary}\footnote{For two definitions of the ideal 
boundary and of its topology, see Definition 7.1 p. 117 of \cite{GH} and chapter 2 of \cite{CDP}, these two definitions being equivalent by  Proposition 
7.4 p. 120 of \cite{GH}.}. It is well known that every isometry $\gamma$ of $(X , d)$ can be extended as a continuous mapping from $X \cup \partial X $ to $X \cup \partial X $ (see for example Proposition 11.2.1 p. 134 of \cite{CDP}). An isometry $\g$ of $(X , d)$
is said to be
\begin{itemize}
\item \emph{elliptic} if, for at least one $x \in X$ (thus for every $x \in X$), the sequence $ k \mapsto \gamma^k x$ is bounded,

\item \emph{parabolic} if, for at least one $x \in X$ (thus for every $x \in X$), the sequence $ k \mapsto \gamma^k x$ admits one and only one accumulation 
point, denoted by $\gamma^\infty$, located on the ideal boundary $\partial X $ ($\gamma^\infty$ does not depend on the choice of $x$),

\item \emph{hyperbolic} if, for at least one $x \in X$ (thus for every $x \in X$), the map $ k \f  \gamma^k x$ is a quasi-isometry from $\Z$ to $X$.
\end{itemize}

The following Theorem is classical.

\begin{theorem}\label{ellparahyp} \emph{(see \cite{CDP}, Th\'eor\`eme 9.2.1 p. 98)}
On a $\delta$-hyperbolic space, every isometry is either elliptic, or parabolic, or hyperbolic.
\end{theorem}

If $\g$ is an hyperbolic (resp. parabolic) isometry, it is a classical result [see for example \cite{CDP}, Proposition 10.6.6 p. 118, (resp. \cite{GH}, 
Th\'eor\`eme 17 in Chapter 8)] that the action of $\g$ on $ X \cup \partial X $ admits exactly two (resp. one) fixed points, 
which are the limits $\gamma^+$ and $\gamma^-$ of $ \gamma^p x$ and $ \gamma^{-p} x$ when $ p \to +\infty$ (resp. which is the limit $\gamma^\infty$ 
of $ \gamma^k x$ when $k \f \pm \infty$).

The following remark is also well known; its proof is trivial (and left to the reader) if one notices that, by Lemma \ref{discret1}, every discrete
subgroup of the group $\text{Isom} (X,d)$ of isometries of $(X,d )$ acts properly on $(X,d )$.

\begin{remark}\label{kpointsfixes}
On a $\delta$-hyperbolic space $(X,d )$, if $\Gamma$ is a discrete subgroup of ${\rm Isom} (X,d)$, then
\begin{itemize}
\item[(i)] an element of $\Gamma^*$ is elliptic if and only if it has torsion; if $\Gamma$ is torsion-free, every $\g \in \Gamma^*$ is either hyperbolic or parabolic;
\item[(ii)] for every $\g \in \Gamma^*$ and every $k \in \Z^*$ such that $\g^k \ne id_X$, $\g$ is hyperbolic (resp. parabolic, resp. elliptic) if and only if $\g^k$ is 
hyperbolic (resp. parabolic, resp. elliptic); moreover, in the cases where $\g $ is hyperbolic or parabolic, then $\g^k$ and $\g$ have the same set of fixed points.
\end{itemize}
\end{remark}

\begin{defis}\label{deplacements}
To each non trivial isometry $\g$ of a $\delta$-hyperbolic space $(X, d)$, one associates:
\begin{itemize}
\item its  \emph{asymptotic displacement} $\ell(\g)$, {\sl i.e.} the limit\footnote{By sub-additivity, this limit exists and, by 
the triangle inequality, it does not depend on the point $x$, see for example \cite{CDP}, 
Proposition 10.6.1 page 118).} (when $k \f + \infty$) of $\frac{1}{k} \  d (x, \g^k x) $,

\item its \emph{minimal displacement} $s(\g) :=\inf_{x \in X} d(x, \g \,x)$.
\end{itemize}
\end{defis}

Notice that $\ell(\g^k ) = |k| \, \ell(\g)$ for every $ k \in \Z$.

The following Lemma is classical (see for instance \cite{CDP}, Proposition 10.6.3, p. 118): 
\begin{lemma}\label{ellpositive}
On a $\delta$-hyperbolic space $(X,d )$, an isometry $\g$ is hyperbolic if and only if $\ell(\gamma) > 0$.
\end{lemma}

The following lemma is proved in \cite{BCGS}, Lemma 8.23 (i):

\begin{lemma}\label{quasigeod}
A non trivial isometry $\g$ of a $\delta$-hyperbolic space $(X, d)$ verifies $\ell(\g ) \le s(\g ) \le \ell(\g ) + \delta$,
\end{lemma}

The following lemma is proved in \cite{BCGS}, Lemma 8.27:
\begin{lemma}\label{geodmin}
For every hyperbolic isometry $\gamma$, for every $ x \in M_{\rm{min}} (\g)$ and any choice of a geodesic segment $ \left[  x , \g  x \right] $ from $x$ to $\g x$, the union $\cup_{p \in \Z}\g^p 
\left(\left[  x , \g \, x \right] \right)$ is a $\g$-invariant local geodesic included in $M_{\rm{min}} (\g)$.
\end{lemma}

\subsection{Elementary subgroups}\label{nilpotents}

In this section we consider \emph{discrete} subgroups $G$ of the isometry group of a Gromov-hyperbolic space $X$. A classification of these groups has been 
sketched by M. Gromov \cite{Gr4} (see also \cite{SUD}, \cite{CCMT}), in terms of their limit set $LG$ ({\sl i.e.} the set of accumulation points of any orbit of the 
action of $G$ on $X$); he classified these groups in the following classes\footnote{Here we only consider groups acting discretely; excluding for instance 
\emph{focal} groups whose action is not discrete.}:

\begin{itemize}
\item {\em elliptic} groups (also said {\em bounded}): finite groups all of whose orbits are bounded;
\item {\em parabolic} (or, according to the original terminology, {\em horocyclic}) groups: infinite groups $G$ such that $\# (LG) = 1$. A parabolic 
group thus only contains parabolic or elliptic elements\footnote{Indeed, if $G$ contains an hyperbolic isometry $g$, its fixed points $g^+$ and $g^-$ are 
accumulation points of the sequence $\left( g^k \, x\right)_{k \in \Z}$, they thus belong to $LG$.}.
\item \emph{lineal} groups: infinite groups $G$ such that $\# (LG) = 2$. A lineal group only contains hyperbolic or elliptic elements (for a proof, see
for instance \cite{C} section 3.4.2).
\item groups of {\em general type}: groups $G$ such that $\# (LG) \ge 3$ (in this case $LG$ is infinite); this is equivalent to say that $G$ contains (at least) two
hyperbolic elements whose fixed points sets are disjoint\footnote{If $G$ contains two hyperbolic elements $ g_1$ and $ g_2$ whose sets of fixed points 
$\text{Fix}(g_1)$ and $\text{Fix}(g_2)$ are disjoint, a trivial consequence is that $\# (LG) > 2$, since $\# (LG)$  contains $\text{Fix}(g_1) \cup \text{Fix}(g_2)$
the converse implication is proved for instance in \cite{C}, Lemma 3.7. The fact that, if $ \# (LG)>2$, then $ LG$ is infinite and uncountable is announced in 
\cite{Gr4},  section 3.5, Theorem p.194; one can find a complete proof in \cite{SUD}, Proposition 6.2.14.}.
\end{itemize}

In the three first cases, the action of $G$ on $X$ is said to be {\em elementary}. By extension, an hyperbolic group will be said to be \emph{elementary} if the
action of $G$ (by left translations) on $G$ (endowed with the algebraic word-metric) is elementary; as, for this action, $LG$ coincides with the ideal boundary 
$\partial G$, an hyperbolic group is \emph{elementary} iff $\# (\partial G) \le 2$.\\
Notice that, if the action of $G$ is elementary, the set $\text{Fix}(g)$ of fixed points of any non elliptic element $ g \in G$ coincides\footnote{Indeed, for any parabolic 
(resp. lineal) group $G$, for every parabolic (resp. hyperbolic) isometry $g \in G$, the fixed points of $g$ are the accumulation points of the sequence
$\left( g^k \,x\right)_{k \in \Z}$, thus $\text{Fix}(g) \subset LG$; furthermore, as $\# \big( \text{Fix}(g)\big) = \# (LG)$, then $\text{Fix}(g) =  LG$.} 
with the limit set $LG$ of $G$.

\medskip
In the three following Propositions, we shall recall basic properties of elementary groups, most of them being immediate corollaries
of the above classification. For proofs or references, see for instance \cite{BCGS}

\begin{prop}\label{actionelementaire} \emph{(Elementary actions)}
Let $ G$ be a discrete subgroup of the isometry group of any Gromov-hyperbolic space $X$, then:
\begin{itemize}
\item[(i)] if $ \g_1, \g_2 \in G $ are two hyperbolic elements with a common fixed point, then they have the same pair of fixed points, {\sl i.e.} $ \{ \g_1^-, \g_1^+ \} 
= \{ \g_2^-, \g_2^+ \} $;

\item[(ii)] if $\g \in G$ is an hyperbolic isometry, the subgroup $G_{\g} = \left\{ g \in G \, : \, g\big( \{ \g^-, \g^+ \} \big) = \{ \g^-, \g^+ \}\right\}$ is the 
maximal subgroup among all virtually cyclic subgroups of $G$ which contain $\g$; if moreover $G$ is torsion-free, then $G_{\g} = \left\{ g \in G \, : \, 
g ( \g^-) =  \g^- \text{ and } g ( \g^+) =  \g^+\right\}$;

\item[(iii)] if $G$ est amenable (\sl{e.g.} virtually nilpotent), then the action of $G$ is elementary; 

\item[(iv)] if $G$ is virtually nilpotent, its non elliptic elements are either all parabolic or all hyperbolic and all have the same set of fixed points.

\item[(v)] if $\g \in G$ is an hyperbolic isometry, for every $g\in G$, the subgroup generated by
$\g$ and $g \g g^{-1}$ is virtually cyclic if and only if the subgroup generated by $\g$ and $g $ is virtually cyclic;

\item[(vi)] if $a$ and $b$ are hyperbolic isometries and if $ \langle a , b  \rangle$ is not virtually cyclic then, for every
$p , q \in \Z^*$, $  a^p$ and $ b^q  $ are hyperbolic and $ \langle a^p, b^q  \rangle$ is not virtually cyclic;

\item[(vii)] if $\g \in G$ is an hyperbolic isometry, any subset $S$ of $G$ such that $ \langle \g, g\rangle$ is 
virtually cyclic for every $g\in S$ generates a virtually cyclic group.
\end{itemize}
\end{prop}

In the co-compact case, one has the following classical results:

\begin{prop}[see \cite{GH}, \cite{CDP}, \cite{BH}]\label{actioncocompacte} 
 Let $G$ be a discrete group of isometries acting co-compactly on a  Gromov-hyperbolic space $X$, then
\begin{itemize}
\item[(i)] $G$ is a finitely generated Gromov-hyperbolic group;
\item[(ii)] $G$ does not contain any parabolic isometry;
\item[(iii)] $G$ is elementary if and only if the space $X$ is elementary;
\item[(iv)] $G$ is elementary if and only if it is virtually cyclic;
\item[(v)] every virtually nilpotent subgroup of $G$ is virtually cyclic.
\end{itemize}
\end{prop}

A corollary of these results is the following elementary lemma:

\begin{lemma}\label{maximalcyclic}
For every group $\Gamma$ and every torsion-free $\g \in \Gamma$, if one of the two following properties is satisfied:
\begin{itemize}
\item[(i)] there exists a proper isometric action of $\Gamma$ on some Gromov-hyperbolic space $(X, d_X)$ such that $\g$ acts as a non-parabolic 
isometry, or
\item[(ii)] there exists a proper co-compact isometric action of $\Gamma$ on some Gromov-hyperbolic space $(X, d_X)$ ,
\end{itemize}
then $\g$ is contained in a unique maximal virtually cyclic subgroup of $\Gamma$ and acts on $(X,d)$ as an hyperbolic isometry.
\end{lemma}

\begin{proof}  A proof of this Lemma can be found in \cite{BCGS} if the hypothesis (i) is assumed. Now, by Proposition \ref{actioncocompacte} (ii), property 
(ii) implies that none of the elements of $\Gamma$ acts on $(X, d_X)$ as a parabolic isometry, hence property (ii) implies property (i), which ends the proof.
\end{proof}

\bibliographystyle{alpha}
\bibliography{bibliography-Margulis-Gromov}

\end{document}